\documentclass[a4paper,11pt]{article}
\usepackage{blindtext}
\usepackage{amsmath,amssymb,amsfonts,latexsym,cancel}
\usepackage{amsthm}
\usepackage{graphicx}
\usepackage{multicol}
\usepackage{multicol}
\usepackage{mathtools}

\usepackage[utf8]{inputenc}
\usepackage[english]{babel}
\usepackage[utf8]{inputenc}
\usepackage{color}
\newtheorem{theorem}{Theorem}[section]

\newtheorem{definition}[theorem]{Definition}

\newtheorem{lemma}[theorem]{Lemma}

\newtheorem{proposition}[theorem]{Proposition}
\newtheorem{remark}[theorem]{Remark}

\numberwithin{equation}{section}

\usepackage{times}
\usepackage{color}
\usepackage{fancyhdr}
\usepackage{pstricks-add}
\usepackage[font={scriptsize,it},textfont=it]{caption}
\usepackage[margin=1in]{geometry}
\input epsf

\begin{document}

\date{}
\author{\quad Rodolfo Viera}

\title{Delone sets that are not rectifiable under Lipschitz co-uniformly continuous bijections} 
%and a quantitative negative answer to Feige's question}
\maketitle

\begin{small}
\noindent{\bf Abstract.} We prove that there exist Delone sets in $\mathbb{R}^d$, $d \geq 2$, which cannot be mapped onto the standard lattice $\mathbb{Z}^d$ by Lipschitz co-uniformly continuous bijections satisfying an asymptotic control on the lower distortion. The impossibility of the unrectifiability crucially uses ideas of Lipschitz regular maps  recently introduced by M. Dymond, V.  Kalu\v{z}a and E. Kopeck\'a.
\end{small}

\vspace{0.5cm}

\section{Introduction}

Motivated by problems in many branches of mathematics ({\em e.g.} metric embedding theory \cite{MatNa, Ost}, geometric group theory \cite{Gromov}, information theory \cite{feige}, mathematical physics of quasicrystals \cite{BaakeGrimm}), over the last years there has been a lot of activity on Lipschitz embeddings of discrete sets. In this work, we focus on a particular aspect of this wide theory, namely the Lipschitz embeddability of Delone subsets of the Euclidean space into the standard lattice. Recall that a {\em Delone set} $\mathcal{D}$ of a metric space $X$ is a subset that is discrete and coarsely dense in a uniform way. This means that there exist positive constants $\sigma,\Sigma$ such that $d (x,y) \geq \sigma$ for all $x \neq y$ in $\mathcal{D}$, 
and for each $z \in X$ there exists $x \in \mathcal{D}$ for which $d (x,z) \leq \Sigma$.\\

Furstenberg and, independently, Gromov, asked whether for every Delone subset of $\mathbb{R}^d$, $d \geq 2$, there 
exists a bi-Lipschitz bijection 
%from it 
onto $\mathbb{Z}^d$ ({\em i.e.} whether every Delone set in $\mathbb{R}^d$ is \textit{bi-Lipschitz 
rectifiable}). Furstenberg was interested in dynamical aspects of this question (see \cite{BK2} for a broader discussion), while Gromov was motivated by instances of geometric group theory \cite{Gromov}. Their question was answered in the negative by Burago and Kleiner \cite{BK} and, independently, by 
McMullen \cite{Mc}. However, their results only yield existence of non bi-Lipschitz rectifiable Delone sets. Concrete examples 
were produced by Cortez and Navas in \cite{CN}. These examples can be constructed with supplementary properties of ``dynamical type''. In particular, they can be built so that they are {\em repetitive}, which means that the translation 
action on the space of Delone sets (endowed with an appropriate Chabauty topology) is minimal. Equivalently, for each $r > 0$, there 
exists $R > 0$ such that every pattern that appears in a ball of radius $r$ actually appears in every ball of radius $R$. Besides the 
mathematical relevance of this property, 
it is worth mentioning that all known examples of real (physical) quasicrystals lead to repetitive Delone sets.\\

In order to answer in the negative the Furstenberg-Gromov' question, Burago and Kleiner in \cite{BK} and McMullen in \cite{Mc} show that the existence of a non-rectifiable Delone set is a consequence of the following result of analytical nature: {\em there exists a bounded away from zero continuous function $\rho:I^2\to\mathbb{R}$, where $I^2:=[-1/2,1/2]^2$, for which the prescribed Jacobian equation}

\begin{equation}\label{nonbiliprealizable}
Jac(F)=\rho \hspace{1cm} a.e,
\end{equation}

{\em has no bi-Lipschitz solution $F:I^2\to\mathbb{R}^2$}; such densities $\rho:I^2\to\mathbb{R}$ which cannot be realizable as the Jacobian of a bi-Lipschitz map are called {\em non-bi-Lipschitz-realizable}. For more details on the prescribed Jacobian equation, we refer for instance to \cite{DM, JM, RY}. We point out that ``almost all" positive functions $\rho\in L^{\infty}$ can be used to construct non-rectifiable Delone sets, as was shown in \cite{V}.\\

Motivated by a fundamental problem in discrete geometry and information theory (see \cite{feige} and the references therein), Dymond, Kalu\v{z}a and Kopeck\'a recently adapted the Burago-Kleiner / McMullen techniques to answer 
in the negative a question raised by Feige in \cite{MatNa}. More precisely, 
in \cite{checos}, they proved the following remarkable fact: 

\begin{center}
{\em Given $d \geq 2$, there is no constant $L$ such that, for all $n$, every subset of $n^d$ points of $\mathbb{Z}^d$ can be 
bijectively mapped into $\{1,\ldots,n\}^d$ by an $L$-Lipschitz map}. 
\end{center}

Their proof relies strongly on the existence of a positive continuous  function $\rho:I^d\to\mathbb{R}$, with $d\geq 2$, for which the generalized ``push-forward equation" 

\begin{equation}\label{non Regular realizable}
F_{\#}(\rho\lambda)=\lambda|_{F(I^d)},
\end{equation}

has no Lipschitz (regular) solution  $F:I^d\to\mathbb{R}^d$, where $\lambda$ is the Lebesgue measure (for a definition of Lipschitz Regular maps, see Section \ref{sec:lipreg} below); observe that equation (\ref{non Regular realizable}) coincides with (\ref{nonbiliprealizable}) whenever $F$ is bi-Lipschitz.

\medskip
%%%%%%%%%%%%%%%%%%%%%%%%%%%%%%%%%%%%%%%%%%%%%%%%%%%%%%%%%%%%%%%%%%%%%%%%%%%%%%%%%%%%%%%%%%%%%%%%%%%%%%%%%%%%%%%%%%%%%%%%%%%%%%%%%%%%%%%%%%%%%%%%%%%%%%%%%%%%%%%%%%%%%%%%%%%%%%%%%%%%%%%%%%%%

\subsection{Notations and some basics on Lipschitz maps}\label{ssection:notations}

Throughout this work we will denote by $||\cdot||$ the supremum-norm in $\mathbb{R}^d$, and we denote by $B(x,r)$ the open ball with center $x\in\mathbb{R}^d$ and radius $r>0$ with this norm; moreover, we write $B(r)$ for the open ball with radius $r>0$ and centred at the origin. Given $\varepsilon>0$ and a bounded set $A\subset\mathbb{R}^d$, let $B(A,\varepsilon)$ be the $\varepsilon$-neighbourhood of $A$. Given two subsets $A,B\subset\mathbb{R}^d$, we denote by $d(A,B)$ the distance between $A$ and $B$ with respect to $||\cdot||$. For two positive numbers $r<R$ and a point $x\in\mathbb{R}^d$, we denote by $\mathsf{Ann}(x,r,R)$ be the annulus $\{y\in\mathbb{R}^d:\ r<||y-x||\leq R\}$. The set $C(I^d)$ will denote the space of real-valued continuous functions defined on the unit square $I^d:=[-1/2,1/2]^d$, with the supremum-norm $||f||_{\infty}:=\max\{|f(x)|:\ x\in I^d\}$. Finally, for $[\cdot]$ we denote the integer part of a real number.

\medskip

We say that a Delone set $\mathcal{D}\subset\mathbb{Z}^2$ satisfies the {\em $2\mathbb{Z}^2$-property} if $2\mathbb{Z}\times\mathbb{Z}$ and $\mathbb{Z}\times 2\mathbb{Z}$ are subsets of $\mathcal{D}$. Finally, we say that a subset $R$ of an integer square $T=([i,i+k]\times [j,j+k])\cap\mathbb{Z}^2$, where $i,j,k\in\mathbb{Z}$, satisfies the {\em $2\mathbb{Z}^2$-property}, if $T\cap (2\mathbb{Z}\times\mathbb{Z})$ and $T\cap (\mathbb{Z}\times 2\mathbb{Z})$ are contained in $R$. 

\medskip

Let $(X,d_1)$ and $(Y,d_2)$ be two metric spaces. A function $f:X\to Y$ is said to be {\em Lipschitz}, if there exists $L>0$ such that for every $x,y\in X$:

\begin{equation*}
d_2(f(x),f(y))\leq Ld_1(x,y).
\end{equation*}

Given two positive numbers $b\leq L$, we say that $f$ is {\em $(b,L)$-bi-Lipschitz} if for every $x,y\in X$:

\begin{equation*}
bd_1(x,y)\leq d_2(f(x),f(y))\leq Ld_1(x,y);
\end{equation*}

moreover $f$ is called $L$-bi-Lipschitz if $b=1/L$.\\

Given an increasing continuous function $\omega: (0,\infty)\to (0,\infty)$, we say that a bijection $f:X\to Y$ is {\em $\omega$-co-uniformly continuous} if for every $x\in X$ there holds

\begin{equation}\label{counif}
f^{-1}(B(f(x),r))\subset B(x,\omega(r));
\end{equation}

a bijection $f:X\to Y$ is called {\em co-uniformly continuous} if is $\omega$-co-uniformly continuous for some increasing continuous function $\omega:(0,\infty)\to (0,\infty)$ (see \cite{JLPS} for further details). Additionally, we say that a $\omega$-co-uniformly continuous map is {\em of order $o(h(r))$} if $\omega(r)=o(h(r))$ (where $o(\cdot)$ denotes the standard Landau's notation).

\medskip

Given a differentiable map $f:U\subset\mathbb{R}^d\to\mathbb{R}^d$, the (determinant) {\em Jacobian of $f$} is denoted by $Jac(f):=\det(Df)$. Recall that by a classical theorem due to Rademacher (see Theorem 3.1.6 in \cite{Fed}), every Lipschitz map $f:U\subset\mathbb{R}^d\to\mathbb{R}^n$ is differentiable almost everywhere. Thus, the expression ``{\em $Jac(f)$ a.e}" make sense for Lipschitz maps from $\mathbb{R}^d$ to itself.\\ 

Given an integrable function $\rho:I^d\to [0,+\infty)$, let $\rho\lambda$ be the measure defined by

\begin{equation*}
\begin{split}
\rho\lambda(A)&:=\displaystyle\int_{A}\rho\lambda,\hspace{0.5cm} \mbox{for every measurable set }A\subset I^d.\\
\end{split}
\end{equation*}

For a measurable map $F:A\to\mathbb{R}^d$ and a measure $\mu$ in $A$, consider the {\em pushforward measure} $F_{\#}\mu$ given by

\begin{equation}
F_{\#}\mu(B):=\mu(F^{-1}(B)) \hspace{0.5cm} \mbox{for every measurable set }B\subset F(A). 
\end{equation}

\medskip
%%%%%%%%%%%%%%%%%%%%%%%%%%%%%%%%%%%%%%%%%%%%%%%%%%%%%%%%%%%%%%%%%%%%%%%%%%%%%%%%%%%%%%%%%%%%%%%%%%%%%%%%%%%%%%%%%%%%%%%%%%%%%%%%%%%%%%%%%%%%%%%%%%%%%%%%%%%%%%%%%%%%%%%%%%%%%%%%%%%%%%%%%%%%%%%%%%%%%%%%%%%%%%%%%%%%%%%%%%%%%%%%%
\subsection{Statement of the result}
The aim of this work is to prove, following \cite{BK} and \cite{checos}, that if we deal with Lipschitz bijections which are co-uniform with a suitable asymptotic control on the lower distortion, then there still exist Delone sets that fail to be rectifiable. 

\medskip

\noindent{\bf Main Theorem.} {\em For each $d \geq 2$, there exist Delone subsets of $\ \mathbb{R}^d$ that admit no Lipschitz co-uniformly continuous bijection with $\mathbb{Z}^d$, with co-uniformity of order $o(r^d)$.}

\medskip
This result extends the well-known bi-Lipschitz case in \cite{BK, Mc, CN}. In particular the Main Theorem implies that there is a Delone set in $\mathbb{R}^d$, $d\geq 2$, which cannot be mapped onto $\mathbb{Z}^d$ by Lipschitz bijections with a {\em H\"{o}lder} co-uniformity. Closely related results have been obtained recently in \cite{checos2}, where it is shown that there is a Delone set $\mathcal{D}\subset\mathbb{R}^d$, $d\geq 2$, for which there is no bijection $f:\mathcal{D}\to\mathbb{Z}^d$ which is $\omega$-homogeneous and $\omega$-co-uniformly continuous for a particular modulus of continuity $\omega:[0,+\infty)\to [0,+\infty)$.
\medskip

It would be interesting to known if there are Delone sets on $\mathbb{R}^d$, $d\geq 2$, that admit no Lipschitz bijections with the standard lattice $\mathbb{Z}^d$, i.e, by avoiding any control in the lower distortion. In addition, since non-rectifiable Delone sets do exist, it would be interesting to know explicit examples of Delone sets that are Lipschitz rectifiable without being bi-Lipschitz rectifiable.

\medskip

\noindent{\bf Sketch of the proof.} Our proof crucially follows the construction proposed in \cite{BK}, \cite{checos} and \cite{checos2}, 
with some mild though crucial changes along the way. Our strategy is described below in the 2-dimensional setting:

\begin{enumerate}
\item[1.- ]Let $\rho:I^2\to\mathbb{R}$ be a positive continuous function such that $8/9\leq \min\rho<\max\rho\leq 1$. As in \cite{BK}, consider a Delone subset $\mathcal{D}_{\rho}$ of $\mathbb{Z}^2$ satisfying the $2\mathbb{Z}^2$-property and emulating the behaviour of $\rho$ at bigger and bigger scales (see Section \ref{ssection:anom_delone} for a precise construction). 

\item[2.-]If there is a Lipschitz bijection $f:\mathcal{D}_{\rho}\to\mathbb{Z}^2$, then $f$ must be ``regular", which means that the preimage of a ball $B$ under this map cannot contain a 
$2 \! \cdot \! \mbox{radius}(B)$-separated set with more than a certain prescribed number of elements (see 
Definition \ref{def Lip reg} and Lemma \ref{def Lip reg 2} below).

\item[3.- ]By renormalizing $\mathcal{D}_{\rho}$ and $f$, and after passing to the limit, we obtain a Lipschitz limit map $F:I^2\to\mathbb{R}^2$ defined on the unit square, which can be also showed to be regular (see Sections \ref{ssection:anom_delone} and \ref{sec:lipreg} for more details).

\item[4.- ]By one of the main results in \cite{checos}, there is an open subset of $\mathbb{R}^2$ whose preimage under $F$ is made up of finitely many disjoint open sets  restricted to which the limit map is bi-Lipschitz; moreover, the number of these sets is uniformly controlled.

\item[5.- ]Finally, from the control on the lower distortion, we show that there is a closed ball $\mathcal{Q}\subset F(I^2)$ such that $F$ satisfies the equation
\begin{equation}\label{genjaceq}
F_{\#}(\rho\lambda)|_{\mathcal{Q}}=\lambda|_{\mathcal{Q}}.
\end{equation}

The proof of that $F$ satisfies equation \eqref{genjaceq} is made by means of a control of the loss of mass in $\mathcal{Q}$ under the renormalization. More precisely, as a consequence of the control of the co-uniformity in \eqref{counif}, we prove that the escape of mass in $\mathcal{Q}$ occurs always close to its boundary.

\medskip

Therefore, if we choose a density map $\rho$ for which there is no a closed ball $\mathcal{Q}$ such that \eqref{genjaceq} has Lipschitz regular solutions, we obtain a Delone set $\mathcal{D}_\rho$ which cannot be Lipschitz rectifiable under co-uniformly continuous bijections of order $o(r^2)$. This function $\rho$ must exist as a consequence of Proposition \ref{non realizable} below (see Theorem 4.1 in \cite{checos}). 
\end{enumerate}

The preceding steps can be summarized as follows: every bijection $f:\mathcal{D}_{\rho}\to\mathbb{Z}^2$ having certain Lipschitz-regularity (e.g, Lipschitz co-uniformly continuous, bi-Lipschitz) must induce certain type of regularity over the continuous density $\rho$. Thus, the desired bad-behaved Delone set $\mathcal{D}_{\rho}$ can be found by choosing $\rho$ lying outside of this regularity class. 

\medskip
%%%%%%%%%%%%%%%%%%%%%%%%%%%%%%%%%%%%%%%%%%%%%%%%%%%%%%%%%%%%%%%%%%%%%%%%%%%%%%%%%%%%%%%%%%%%%%%%%%%%%%%%%%%%%%%%%%%%%%%%%%%%%%%%%%%%%%%%%%%%%%%%%%%%%%%%%%%%%%%%%%%%%%%%%%%%%%%%%%%%%%%%%%%%%%%%%%%%%%%%%%%%%%%%%%%%%%%%%%

\section{Constructing a Delone set from an anomalous density.}\label{ssection:anom_delone}

In this section we explain how to produce a Delone set from a bounded away from zero density, as in \cite{BK} and \cite{checos} (see also \cite{checos2}). To simplify computations, notations and figures, we will restrict ourself to the 2-dimensional case; the higher dimensional case follows analogously.\\

For our purposes, we will consider a continuous density $\rho: I^2\to\mathbb{R}$  such that $8/9\leq \min\rho<\max\rho\leq 1$. Let $(l_n)_{n\in\mathbb{N}}$ and $(m_n)_{n\in\mathbb{N}}$ be two sequences of even positive integers, where $l_n$ is a multiple of $m_n$ and such that $l_n,m_n$ and $(l_n/m_n)\longrightarrow +\infty$. Moreover, assume that $l_n$ divides $l_{n+1}$ for every $n\in\mathbb{N}$; this last hypothesis will be used later in the proof of the Main Theorem. Let $(S_n)_{n\in\mathbb{N}}$ be a sequence of disjoint squares with sides parallel to the coordinate axes, vertices having even integer coordinates and with side-length equal to $l_n$. For each $n\in\mathbb{N}$, let $(T_{n,i})_{i=1}^{m_n^2}$ be a subdivision of $S_n$ by $m_n^2$ squares with sides parallel to the coordinate axes and side-length equal to $l_n/m_n$ (see Figure 1). Finally, let $\phi_n:\mathbb{R}^2\to\mathbb{R}^2$ be an affine linear map sending the square $S_n$ onto the unit square $I^2$.

\begin{figure}[h!]
\psset{xunit=0.6cm,yunit=0.6cm,algebraic=true,dimen=middle,dotstyle=o,dotsize=5pt 0,linewidth=1.6pt,arrowsize=3pt 2,arrowinset=0.25}
\begin{pspicture}(-4,0)(18,11)
\pspolygon[linewidth=1pt](1,1)(4,1)(4,4)(1,4)
\pspolygon[linewidth=1pt](6,1)(14,1)(14,9)(6,9)
\pspolygon[linewidth=1pt, fillstyle=solid, fillcolor=gray](2,2)(3,2)(3,3)(2,3)
\psline[linewidth=1pt](1,1)(4,1)
\psline[linewidth=1pt](4,1)(4,4)
\psline[linewidth=1pt](4,4)(1,4)
\psline[linewidth=1pt](1,4)(1,1)
\psline[linewidth=1pt](6,1)(14,1)
\psline[linewidth=1pt](14,1)(14,9)
\psline[linewidth=1pt](14,9)(6,9)
\psline[linewidth=1pt](6,9)(6,1)
\psline[linewidth=1pt](2,1)(2,4)
\psline[linewidth=1pt](3,1)(3,4)
\psline[linewidth=1pt](1,3)(4,3)
\psline[linewidth=1pt](1,2)(4,2)
\psline[linewidth=1pt](6,3)(14,3)
\psline[linewidth=1pt](6,5)(14,5)
\psline[linewidth=1pt](6,7)(14,7)
\psline[linewidth=1pt](8,1)(8,9)
\psline[linewidth=1pt](10,1)(10,9)
\psline[linewidth=1pt](12,1)(12,9)
\rput[tl](2.120818292403509,0.6386127491732377){$S_n$}
\rput[tl](9.616561529359915,0.6211401542152974){$S_{n+1}$}
\rput[tl](12.2,10.7){$\dfrac{l_{n+1}}{m_{n+1}}$}
\rput[tl](1.1,5.6){$\dfrac{l_n}{m_n}$}
\psline[linewidth=0.4pt](15,1)(16,1)
\psline[linewidth=0.4pt](16,1)(16,4)
\psline[linewidth=0.4pt](16,6)(16,9)
\psline[linewidth=0.4pt](16,9)(15,9)
\rput[tl](15.6,5.3){$l_{n+1}$}
\psline[linewidth=0.4pt](12.18,10)(12,10)
\psline[linewidth=0.4pt](12,10)(12,9.5)
\psline[linewidth=0.4pt](13.806937530175594,9.992750368530654)(14,10)
\psline[linewidth=0.4pt](14,10)(14,9.5)
\psline[linewidth=0.4pt](0.5832299361047588,1.0055372432899843)(0.01747259495794029,1)
\psline[linewidth=0.4pt](0.01747259495794029,1)(0,2)
\psline[linewidth=0.4pt](0,3)(0,4)
\psline[linewidth=0.4pt](0,4)(0.6007025310626992,3.993350981097778)
\rput[tl](-0.18,2.7){$l_n$}
\psline{->}(3,5)(2.5,2.5)
\rput[tl](3,6){$T_{n,i}$}
\end{pspicture}
\caption{The squares $S_n$ and $S_{n+1}$, and a square $T_{n,i}$.}
\end{figure}
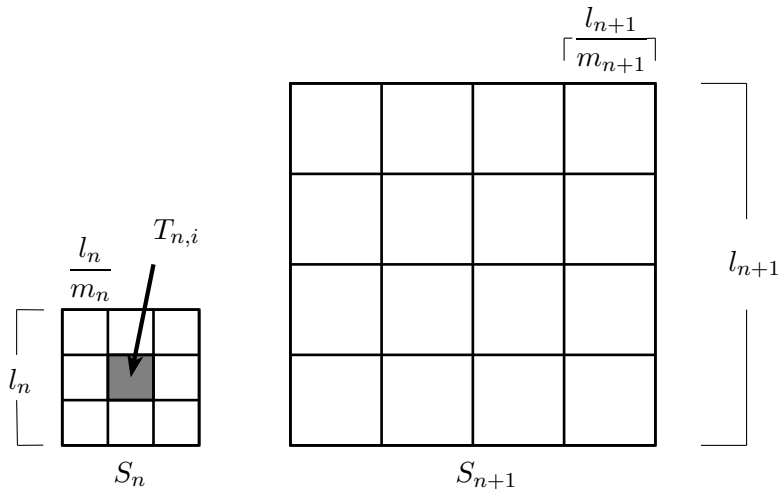

We build a Delone set $\mathcal{D}_{\rho}\subset\mathbb{Z}^2$ which ``emulates" the behaviour of $\rho$ as follows: in each square $T_{n,i}$ we put $\left[\int_{T_{n,i}}\rho\circ\phi_n d\lambda\right]$ points with integer coordinates in such a way that each $T_{n,i}$ satisfies the $2\mathbb{Z}^2$-property (see Figure 2) and such that $\partial T_{n,i}\cap\mathbb{Z}^2\subset\mathcal{D}_{\rho}$; notice that this is possible since $8/9\leq \min\rho<\max\rho\leq 1$. Outside of $\cup_{n\in\mathbb{N}}S_n$, put one point in each integer coordinate. This construction provides a set $\mathcal{D}_{\rho}\subset\mathbb{Z}^2$ which is actually a Delone set satisfying the $2\mathbb{Z}^2$-property since $8/9\leq \min\rho<\max\rho\leq 1$.

\begin{figure}[h!]
\psset{xunit=0.7cm,yunit=0.7cm,algebraic=true,dimen=middle,dotstyle=o,dotsize=5pt 0,linewidth=1.6pt,arrowsize=3pt 2,arrowinset=0.25}
\begin{pspicture}(-8.5,0)(12,6)
\begin{scriptsize}
\psdots[dotstyle=*](1,1)
\psdots[dotstyle=*](1,2)
\psdots[dotstyle=*](1,3)
\psdots[dotstyle=*](1,4)
\psdots[dotstyle=*](1,5)
\psdots[dotstyle=*](2,5)
\psdots[dotstyle=*](3,5)
\psdots[dotstyle=*](4,5)
\psdots[dotstyle=*](5,5)
\psdots[dotstyle=*](5,4)
\psdots[dotstyle=*](5,3)
\psdots[dotstyle=*](5,2)
\psdots[dotstyle=*](5,1)
\psdots[dotstyle=*](2,1)
\psdots[dotstyle=*](3,1)
\psdots[dotstyle=*](4,1)
\psdots[dotstyle=*](3,2)
\psdots[dotstyle=*](3,3)
\psdots[dotstyle=*](3,4)
\psdots[dotstyle=*](2,3)
\psdots[dotstyle=*](4,3)
\psdots[dotstyle=*](4,4)
\end{scriptsize}
\end{pspicture}
\caption{A possible configuration of $\mathcal{D}_{\rho}\cap T_{n,i}$.}
\end{figure}
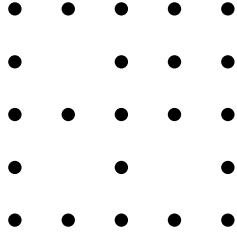

%%%%%%%%%%%%%%%%%%%%%%%%%%%%%%%%%%%%%%%%%%%%%%%%%%%%%%%%%%%%%%%%%%%%%
%%%%%%%%%%%%%%%%%%%%%%%%%%%%%%%%%%%%%%%%%%%%%%%%%%%%%%%%%%%%%%%%%%%%%%

\section{Rescaling up to the limit and Lipschitz regularity}\label{sec:lipreg}

In this section we show, after renormalization and passing to the limit, that a Lipschitz bijection $f:\mathcal{D}_{\rho}\to\mathbb{Z}^2$ induces a Lipschitz regular map from the unit square. We start this section with some basic background and recent results on Lipschitz regular maps. For additional information about these mappings, we refer to \cite{DS}.

\vspace{0.2cm}

\begin{definition}\label{def Lip reg}
Let $X$ and $Y$ be two metric spaces. We say that a Lipschitz map $f:X\to Y$ 
is Lipschitz regular if there is a constant $C\in\mathbb{N}$ such that for 
%every $r>0$ and 
every ball $B\subset Y$ of radius $r  > 0$, the set $f^{-1}(B)$ can be covered by at most $C$ balls of radii $Cr$. The 
smallest such $C$ (that works for every $r > 0$) is called the regularity constant of $f$, and is denoted by $\mathsf{Reg}(f)$.
\end{definition} 

\vspace{0.2cm}

A useful equivalent interpretation of this definition is provided by the next lemma. 
(The proof is straightforward and is left to the reader.)

\vspace{0.2cm}

\begin{lemma}\label{def Lip reg 2}
A Lipschitz map $f:X\to Y$ is Lipschitz regular if and only if there is a constant $C \in \mathbb{N}$ such that for 
%every $r>0$ and 
every ball $B\subset Y$ of radius $r > 0$, the set $f^{-1}(B)$ does not contain a $Cr$-separated set with more than $C$ elements. 
If this is the case, then $\mathsf{Reg} (f) \leq C$. Conversely, if $f$ is Lipschitz regular, then $C$ can be taken as being equal to $2\mathsf{Reg}(f)$.
\end{lemma}

\vspace{0.2cm}

One of the main results of \cite{checos} (namely, Theorem 2.10 therein) is that every Lipschitz regular map defined on a bounded 
region of $\mathbb{R}^d$ can be ``densely decomposed" into bi-Lipschitz pieces, as stated below. 

\vspace{0.2cm}

\begin{theorem}\label{bilip decom 1}
Let $U\subset\mathbb{R}^d$ be a nonempty open set. If $f :\overline{U}\to\mathbb{R}^d$ is a Lipschitz regular map, then there exist pairwise 
disjoint open sets $(A_n)_{n\in\mathbb{N}}$ in $U$ such that $\bigcup_{n\in\mathbb{N}}A_n$ is dense in $\overline{U}$ and, for each 
$n\in\mathbb{N}$, the map $f|_{A_n}$ is bi-Lipschitz  with lower bi-Lipschitz constant $b = b \, ( \mathsf{Reg}(f))$.
\end{theorem}

\vspace{0.2cm}

As a consequence of Theorem \ref{bilip decom 1}, it is showed in \cite{checos} that the image by 
%bounded region where 
a Lipschitz regular map contains an open set whose preimage is made of a controlled number of open subsets where the map is bi-Lipschitz, with 
lower bi-Lipschitz constant depending only on the regularity constant. This is stated below, and corresponds to Proposition 2.15 in \cite{checos}.

\vspace{0.2cm}

\begin{proposition}\label{bilip decom 2}
Let $U\subset\mathbb{R}^d$ be a nonempty open  set. If $f:\overline{U}\to\mathbb{R}^d$ is a Lipschitz regular map, then there exist a 
nonempty open set $T\subset f(\overline{U})$, an integer $N\in\{1,\ldots,\mathsf{Reg}(f)\}$, and pairwise disjoint open sets $W_1,\ldots W_N\subset\overline{U}$, 
such that $f^{-1}(T)=\displaystyle\bigcup_{i=1}^N W_i$ and, for each $1\leq i\leq N$, the map $f|_{W_i}: W_i \to T$ is a bi-Lipschitz homeomorphism, 
with lower bi-Lipschitz constant $b=b(\mathsf{Reg}(f))$. Actually, one may take $b = \frac{1}{2 \mathsf{Reg}(f)^2}$.
\end{proposition} 

\vspace{0.3cm}

The following result deals with maps defined on discrete sets of points. It 
asserts that every Lipschitz bijection defined on a Delone set that satisfies the $2\mathbb{Z}^2$-property onto the 
integer lattice must be Lipschitz regular. Actually, as we will see along the proof, the $2\mathbb{Z}^2$-property may be replaced by any 
other property ensuring that densities of points in large balls are everywhere bounded from below (away from zero). Given $x\in\mathbb{R}^2$, 
$r>0$ and a subset $\mathcal{L}$ of $\mathbb{R}^2$, we will denote the set $B(x,r)\cap\mathcal{L}$ by $B_{\mathcal{L}}(x,r)$.

\vspace{0.2cm}

\begin{proposition}\label{Lip reg 1}
Let $\mathcal{D}\subset\mathbb{Z}^2$ be a Delone set satisfying the $2\mathbb{Z}^2$-property. If $f \!: \mathcal{D}\to\mathbb{Z}^2$ 
is an $L$-Lipschitz bijection, then $f$ is Lipschitz regular, with $\mathsf{Reg}(f)\leq\max\{2, C \, (L+1)^2\}$ for some universal constant $C > 0$.
\end{proposition}

\begin{proof}
Let $y\in\mathbb{Z}^2$ and $r>0$. Consider $\Gamma\subset f^{-1}(B(y,r))$ a maximal $2r$-separated set, and write 
$\Gamma=\{x_1,\ldots x_{|\Gamma|}\}$. Then, by the $L$-Lipschitz condition we have that 
\begin{equation*}
f\left(\bigcup_{i=1}^{|\Gamma|}B_{\mathcal{D}}(x_i,r)\right)\subset B_{\mathbb{Z}^2}(y,r+rL).
\end{equation*}
Observe that, for $i=1,\ldots,|\Gamma|$, the open balls $B(x_i,r)$ are pairwise disjoint.
Since $f$ is a bijection we obtain that, for a certain constant $C_1 \geq 1$, 
\begin{equation}\label{est-L-reg}
\displaystyle\sum_{i=1}^{|\Gamma|}\left|B_{\mathcal{D}}(x_i,r)\right| 
= \left|f\left(\bigcup_{i=1}^{|\Gamma|}B_{\mathcal{D}}(x_i,r)\right)\right| 
\leq \big| B_{\mathbb{Z}^2}(y,r+rL) \big| \leq C_1 r^2 (L+1)^2.
\end{equation}
Now, by the $2\mathbb{Z}^2$-property, the cardinality 
$|B_{\mathcal{D}}(x_i,r)|$ is at least 
$C_2r^2 $ for another universal constant $C_2 > 0$. (The value of $C_2$ can be taken as $8/9 - \varepsilon$ 
provided $r$ is large enough.) Thus, by (\ref{est-L-reg}),
\begin{equation*}
C_2r^2 |\Gamma| \leq C_1r^2(L+1)^2.
\end{equation*}
We hence conclude that $|\Gamma|\leq C_1 (L+1)^2 / C_2$. Therefore, by Lemma \ref{def Lip reg 2},  
$f$ is Lipschitz regular with regularity constant at most $\max\{2,C_1 (L+1)^2 / C_2\}$.
\end{proof}

From now on, let $\mathcal{P}_n:=\mathcal{D}_{\rho}\cap S_n$, where $\mathcal{D}_{\rho}$ and $(S_n)_{n\in\mathbb{N}}$ are the Delone set and the sequence of squares given in \S \ref{ssection:anom_delone}, respectively. Let $\phi_n:\mathbb{R}^2\to\mathbb{R}^2$ be the 
homothety defined in Section \ref{ssection:anom_delone}, and define $\mathcal{R}_n:=\phi_n(\mathcal{P}_n)$.

\medskip

Assume there is an $L$-Lipschitz bijection $f : \mathcal{D}_{\rho} \to \mathbb{Z}^2$.  As in \cite{BK} and \cite{checos}, we proceed to normalize $f$ to each square $\mathcal{P}_n$, that is, to consider the map $f_n:\mathcal{R}_n\to\frac{1}{l_n}\mathbb{Z}^2$ defined by

\begin{equation}\label{normalization}
f_n(x):=\frac{1}{l_n}(f\circ\phi_n^{-1}(x)-f\circ\phi_n^{-1}(\tilde{x}_n)),
\end{equation}

where $\tilde{x}_n\in\mathcal{R}_n$ is some base point. Notice that for each $n \geq 1$ the map $f_n$ is Lipschitz regular with $\mathsf{Lip} (f_n) \leq L$ and $\mathsf{Reg}(f_n)=\mathsf{Reg}(f)$. By Kirszbraun's extension theorem\footnote{Actually, we do not really need to keep the same Lipschitz constant $L$ for the extension map, 
but just another (larger) constant that depends only on $L$, and a weaker form of Kirszbraum's theorem proving this is much easier 
to establish.} (see, for instance, Theorem 2.10.43 in \cite{Fed}), each function $f_n$ can be extended 
to an $L$-Lipschitz map 

\begin{equation}{\label{extension}}
\widehat{f}_n:I^2\to\mathbb{R}^2.
\end{equation}

By the Arzel\'a-Ascoli's theorem, there exists a 
subsequence $(\widehat{f}_{n_k})_{k\in\mathbb{N}}$ of $(\widehat{f}_n)_{n\geq 1}$, converging to an $L$-Lipschitz map 
$F:I^2\to\mathbb{R}^2$; from now on, the subsequence $(\widehat{f}_{n_k})_{k\geq 1}$ will be just denoted  $(\widehat{f}_{n})_{n\geq 1}$. As we next show, the Lipschitz regularity is inherited from $f$ to $F$.

\vspace{0.2cm}

\begin{proposition}\label{Lip reg 2}
The map $F:I^2\to\mathbb{R}^2$ built above is Lipschitz regular, with $\mathsf{Reg}(F) \leq 34 \ \mathsf{Reg}(f)$.
\end{proposition}

\begin{proof}
Let $y\in F(I^2)$ and $r>0$. Consider a maximal $2\mathsf{Reg}(f)r$-separated set $\Gamma=\{x_1,\ldots,x_{|\Gamma|}\}$ 
contained in $F^{-1}(B(y,r))$. Given 
$$0 < \varepsilon < \min \left\{2\mathsf{Reg}(f)r\left(1-\sqrt{\frac{2-\sqrt{2}}{2+\sqrt{2}}}\right), d(\Gamma, \partial F^{-1}(B(y,r))) \right\},$$ 
by the convergence of $\widehat{f}_n$ to $F$, there is a positive integer $n_0=n_0(\varepsilon)$ such that,  
for every $n\geq n_0$, there exist $p_1,\ldots,p_{|\Gamma|}\in \mathcal{R}_n$ for which the following hold: 

\begin{itemize}
\item for every $i=1,\ldots,|\Gamma|$, we have that $||p_i-x_i||<\varepsilon/2$,
\item the set $\Gamma_n:=\{p_1,\ldots,p_{|\Gamma|}\}$ is contained in $F^{-1}(B(y,r))$ and,
\item $f_n(\Gamma_n)\subset B(y,r)$.
\end{itemize}

Observe that $\Gamma_n$ is $(2\mathsf{Reg}(f)r-\varepsilon)$-separated, since $\Gamma$ is $2\mathsf{Reg}(f)r$-separated.\\ 

We will delete some points in $\Gamma_n$ in an appropriate way in order to obtain a set $\Gamma_n'\subset f_n^{-1}(B(y,r))$ 
that is $2\mathsf{Reg}(f)r$-separated and such that $|\Gamma_n'|\geq |\Gamma|/17$. By Lemma \ref{def Lip reg 2}, this will imply that 
$$| \Gamma | \leq 17 \, | \Gamma_n' | \leq 34  \, \mathsf{Reg} (f_n)=34 \,\mathsf{Reg}(f)$$
hence $F$ is Lipschitz-regular with $\mathsf{Reg} (F) \leq 42 \, \mathsf{Reg} (f)$.\\

To build the set $\Gamma_n'$, we consider the angle 
$$\alpha = \arctan \left( \frac{2\mathsf{Reg}(f)r-\varepsilon} {2\mathsf{Reg}(f)r} \right) \geq \arctan \left( \sqrt{\frac{2-\sqrt{2}}{2+\sqrt{2}}} \right) = \frac{\pi}{8},$$
where the inequality follows from the condition
$$\varepsilon <2\mathsf{Reg}(f)r\left(1- \sqrt{\frac{2-\sqrt{2}}{2+\sqrt{2}}}\right).$$ 

This is the angle that appears in Figure 3 below. In the area depicted in black, no pair of points in $\Gamma_n$ is at distance $> 2\mathsf{Reg}(f)r - \varepsilon$. 
The same happens in a similar region with angle $\pi / 8$. Since 16 of these polygonal regions cover exactly the anular region between 
a square of radius $2\mathsf{Reg}(f)r-\varepsilon$ and another of radius $2\mathsf{Reg}(f)r$ (with the same center), we deduce -by the pigeonhole principle- 
that no more than $16$ points of $\Gamma_n$ in this annular region can be $(2\mathsf{Reg}(f)r-\epsilon)$-separated. \\

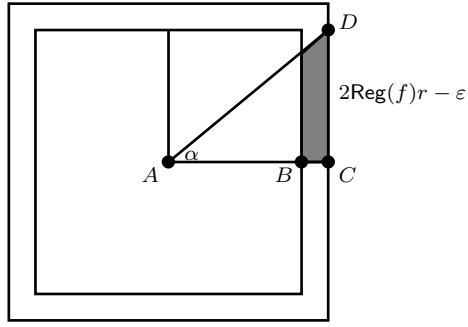
\begin{figure}[h!]
\centering
\psset{xunit=0.7cm,yunit=0.7cm,algebraic=true,dimen=middle,dotstyle=o,dotsize=5pt 0,linewidth=2pt,arrowsize=3pt 2,arrowinset=0.25}
\begin{pspicture}(0,0)(10,8)
\pspolygon[linewidth=1pt](2,1)(8,1)(8,7)(2,7)
\pspolygon[linewidth=1pt](2.5,1.5)(7.5,1.5)(7.5,6.5)(2.5,6.5)
\pspolygon[linewidth=1pt, fillcolor=gray, fillstyle=solid](7.5,4)(8,4)(8,6.5)(7.5,6.05)
\psline[linewidth=1pt](5,4)(8,4)
\psline[linewidth=1pt](5,4)(5,6.5)
\psline[linewidth=1pt](5,4)(8,6.5)
\begin{scriptsize}
\psdots[dotstyle=*](5,4)
\psdots[dotstyle=*](7.5,4)
\psdots[dotstyle=*](8,4)
\psdots[dotstyle=*](8,6.5)
\rput[tl](5.3,4.2){$\alpha$}
\rput[tl](4.5,3.9){$A$}
%\rput[tl](5.1,3.7){$2\mathsf{Reg}(f)r$}
\rput[tl](7,3.9){$B$}
\rput[tl](8.2,3.9){$C$}
\rput[tl](8.2,5.5){$2\mathsf{Reg}(f)r-\varepsilon$}
\rput[tl](8.2,6.8){$D$}
\end{scriptsize}
\end{pspicture}
\caption{In the figure, $\overline{AC}=2\mathsf{Reg}(f)r,\ \overline{BC}=\varepsilon$ and $\overline{CD}=2\mathsf{Reg}(f)r-\varepsilon$. In the black region there is no a pair of points in $\Gamma_n$ at distance $>2\mathsf{Reg}(f)r-\varepsilon$.}
\end{figure}

Now, for each $i\in\{1,\ldots,|\Gamma|\}$, let $\Gamma_n^i$ be the set of all points $p\in\Gamma_n$ such 
that $2\mathsf{Reg}(f)r-\varepsilon\leq ||p_i-p||\leq 2\mathsf{Reg}(f)r$. We have shown that this set contains at most $16$ points. We erase 
those corresponding to $p_1$, then those corresponding to the $p_i$ with minimal index that survive after 
the first deletion ($i \geq 2$), and so on. At the end, we get the subset $\Gamma_n'$ with the desired properties.
\end{proof}

\vspace{0.2 cm} 

%%%%%%%%%%%%%%%%%%%%%%%%%%%%%%%%%%%%%%%%%%%%%%%%%%%%%%%%%%%%%%%%%%%%%%%%%%%%%%%%%%%%%%%%%%%%%%%%%%%%%%%%%%%%%%%%%%%%%%%%%%%%%%%%%%%%%%%%%%%%%%%%%%%%%%%%%%%%%%%%%%%%%%%%%%%%%%%%%%%%%%%%%%%%%%%%%%%%%%%%%%%%%%%%%%%%%%%
\section{Proof of the Main Theorem}

This section is dedicated to the proof of the Main Theorem. Exploiting the bi-Lipschitz decomposition of Lipschitz regular maps recently introduced by Dymond, Kalu\v{z}a and Kopeck\'a in \cite{checos} together the co-uniformity of $f:\mathcal{D}_{\rho}\to\mathbb{Z}^2$, in this section we show that the Lipschitz-rectifiability of $\mathcal{D}_{\rho}$ under co-uniformly continuous bijections of order $o(r^2)$ induces certain regularity over the positive continuous density $\rho:I^2\to\mathbb{R}$. Thus, the existence of a non-Lipschitz-rectifiable Delone set $\mathcal{D}_{\rho}$ will be a consequence of the existence of a continuous function $\rho$ which does not belong to this regularity class.

\medskip

Let $(A_n)_{n\in\mathbb{N}}$ be a basis for the topology of the unit square $I^2$. As in \cite{checos}, let $\mathcal{E}_{C,L,n}$ be the set of positive continuous functions $\rho:I^2\to\mathbb{R}$ for which the following holds:  there are pairwise disjoint open sets $Y_1,\ldots, Y_N\subset I^2$, $Y_1=A_n$, where $1\leq N\leq C$, an open set $V\subset\mathbb{R}^2$, and a family of $(b(C),L)$-bi-Lipschitz homeomorphisms $F_i:Y_i\to V$ such that 

\begin{equation}\label{linear comb jacobians}
\rho(y)=|Jac(F_1)(y)|-\displaystyle\sum_{i=2}^{n}\rho(F_{i}^{-1}\circ F_1)(y)|Jac(F_{i}^{-1}\circ F_1)(y)| \hspace{1cm} \mbox{a.e in } Y_1.
\end{equation}

In \cite{checos} it is shown that ``almost all" positive continuous functions do not have a bi-Lipschitz decomposition as in (\ref{linear comb jacobians}). This corresponds to Theorem 4.1 in \cite{checos}.

\medskip

\begin{proposition}[Dymond, Kalu\v{z}a and Kopeck\'a, 2018]\label{non realizable}
A generic positive function $\rho:I^2\to\mathbb{R}$ does not belong to $\bigcup_{C,L,n\in\mathbb{N}}\mathcal{E}_{C,L,n}$.
\end{proposition}

\medskip

From now on let $\rho$ be a positive continuous function such that $8/9\leq\min\rho<\max\rho\leq 1$ as in Proposition \ref{non realizable} and let $\mathcal{D}_{\rho}\subset\mathbb{Z}^2$ be the corresponding Delone set constructed as in Section \ref{ssection:anom_delone}. Assume that there is an $L$-Lipschitz, $\omega$-co-uniformly continuous bijection $f:\mathcal{D}_{\rho}\to\mathbb{Z}^2$ for some increasing continuous function $\omega:[0,\infty)\to [0,\infty)$ such that $\omega(r)=o(r^2)$ and let $F:I^2\to\mathbb{R}^2$ be the (limit) Lipschitz regular map obtained as in Section \ref{sec:lipreg}. By Proposition \ref{bilip decom 2} there exist a non-empty open set $W\subset F(I^2)$ and open disjoint subsets $V_1,\ldots, V_N\subset I^2$, where $N\leq \mathsf{Reg}(F)$, such that $\bigcup_{i=1}^{N}V_i=F^{-1}(W)$ and, for each $i\leq i\leq N$, the map $F|_{V_i}:V_i\to W$ is bi-Lipschitz, with lower bi-Lipschitz constant $b=\frac{1}{2\mathsf{Reg}(F)^2}$. Thereupon, the Main Theorem is a consequence of the next proposition.

\medskip

\begin{proposition}\label{local non realizable}
For $\rho, f, F$ and $W$ as in the previous paragraph, let $\mathcal{Q}\subset W$ be a closed ball centred at a point $y_m\in (1/l_m)\mathbb{Z}^2$, for some positive integer $m$. Then we have 

\begin{equation}\label{local non realizable equation}
F_{\#}(\rho\lambda)|_{\mathcal{Q}}=\lambda|_\mathcal{Q}.
\end{equation}
\end{proposition}

\begin{proof}[Proof of Main Theorem from Proposition \ref{local non realizable}]
For every $i=1,\ldots,N$ write $F_i:=F|_{V_i}$ and let $n$ be a natural number such that $A_n\subset V_1\cap F^{-1}(\mathcal{Q})$; besides, for each $i=2,\ldots, N$ denote by $A_{i,n}:=F^{-1}(F(A_n))\cap V_i$ and $A_{1,n}:=A_n$. From the equation (\ref{local non realizable equation}) we have that

\begin{equation*}
\displaystyle\sum_{i=1}^{N}\int_{A_{i,n}}\rho d\lambda=\lambda(F(A_n)).
\end{equation*}

By a change of variable and the Euclidean Area formula for bi-Lipschitz maps, the previous equation can be rewritten as

\begin{equation*}
\displaystyle\int_{A_n}\sum_{i=1}^{N}\rho(F_i^{-1}\circ F_1)|Jac(F_i^{-1}\circ F_1)| d\lambda=\displaystyle\int_{A_n}|Jac(F_1)|d\lambda,
\end{equation*}

which is equivalent to the equation

\begin{equation*}
\rho(y)=|Jac(F_1)(y)|-\displaystyle\sum_{i=2}^{n}\rho(F_{i}^{-1}\circ F_1)(y)|Jac(F_{i}^{-1}\circ F_1)(y)| \hspace{1cm} \mbox{a.e in } A_n.
\end{equation*}

Thus, we conclude that $\rho\in\mathcal{E}_{C,L,n\in\mathbb{N}}$ for $C=1/2\mathsf{Reg}(F)^2$, which contradicts the choice of $\rho$. Therefore, $\mathcal{D}_{\rho}$ cannot be mapped onto $\mathbb{Z}^2$ by Lipschitz bijections, as announced.
\end{proof} 
%%%%%%%%%%%%%%%%%%%%%%%%%%%%%%%%%%%%%%%%%%%%%%%%%%%%%%%%%%%%%%%%%%%%%%%%%%%%%%DEMOSTRACIÓN PROPOSICIÓN 3.3%%%%%%%%%%%%%%%%%%%%

\subsection{Mass-loss control under renormalization}

In what follows we prove Proposition \ref{local non realizable}. Let $\mathcal{Q}\subset W$ be a closed ball such that $d(\mathcal{Q},\partial W)>0$ and define $\mathcal{H}_i:=F^{-1}(\mathcal{Q})\cap V_i$ (recall that we consider ``balls" for the sup-norm in $\mathbb{R}^2$) and $\mathcal{H}:=\cup_{i=1}^N \mathcal{H}_i$; in addition, choose $\varepsilon>0$ such that the closure of the $\varepsilon$-neighbourhood $B(\mathcal{Q},\varepsilon)$ of $\mathcal{Q}$ is contained in $W$. We start by showing that from a large-enough $n\in\mathbb{N}$ the points in $\mathcal{R}_n\cap\mathcal{H}$ are mapped under $f_n$ into $B(\mathcal{Q},\varepsilon)$, and that $f_n^{-1}(\mathcal{Q}\cap (1/l_n)\mathbb{Z}^2)$ is completely contained in the $\varepsilon$-neighbourhood of $\mathcal{H}$ (with $f_n$ defined as in (\ref{normalization})); nevertheless, there may exist some points in $\mathcal{Q}\cap(1/l_n)\mathbb{Z}^2$ which do not have a pre-image under $f_n$, producing loss of mass under the  renormalization (see Figure 4).

%%%%%%%%%%%%%%%%%%%%%%%%%%%%%%%%%%%%%%%%%%%%%%%%%%%%%%%%%%%%%%%%%%%%%%%%%%%%%%%%%%%%%%%%%%%%%%%%%%%%%%%%%%%%%%%%%%%%%%%%%%%%%%%%%%%%%%%%%%%%%%%%%%%%%
%AQUI ESTA LA FIGURA 4
 
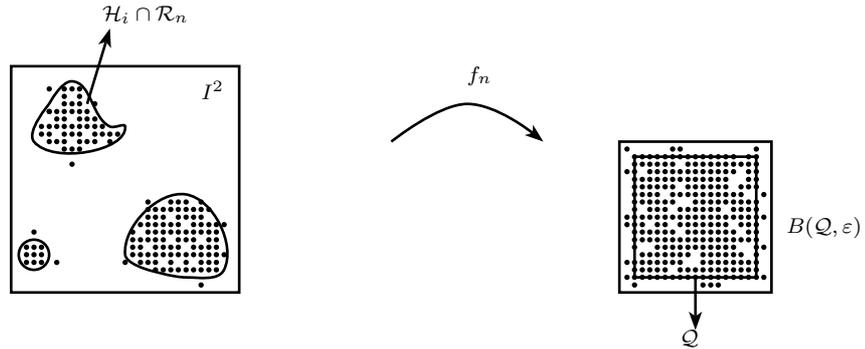
\begin{figure}[h!]\label{fig:lossmass}
\psset{xunit=1cm,yunit=1cm,algebraic=true,dimen=middle,dotstyle=o,dotsize=5pt 0,linewidth=2pt,arrowsize=3pt 2,arrowinset=0.25}
\begin{pspicture}(-2,0)(12,5.5)
\pspolygon[linewidth=1pt](2,1)(2,4)(5,4)(5,1)
\pspolygon[linewidth=1pt](10,3)(10,1)(12,1)(12,3)
\pspolygon[linewidth=1pt](10.2,2.8)(10.2,1.2)(11.8,1.2)(11.8,2.8)

\pscurve[linewidth=1pt](2.5,3.5)(2.8,3.8)(3.3,3.2)(3.5,3.2)(2.3,3)(2.5,3.5)
\pscurve[linewidth=1pt](3.8,1.3)(4.3,1.2)(4.8,1.3)(4.3,2.3)(3.5,1.5)(3.8,1.3)
\pscircle[linewidth=1pt](2.3,1.5){0,2}
\begin{scriptsize}
%%%%%%%%%%%%%%%%%%%%%%%%%%%%%%%%%%%%%%%%%%%%%%%%%%%%%%%%%%%%%%%%%%%%%%%%%%%%%%%%%%%%%%%%%%%%%%%%%%%%%%%%%%%%%%%%%%%%%%%%%%%%%%%%%%%%%%%%%%%%%%%%%%%%
\pscurve[linewidth=1pt]{->}(7,3)(8,3.5)(9,3)
\rput[tl](8,4){$f_n$}
\psline[linewidth=1pt]{->}(3,3.5)(3.3,4.5)
\rput[tl](3.2,4.8){$\mathcal{H}_i\cap\mathcal{R}_{n}$}
%%%%%%%%%%%%%%%%%%%%%%%%%%%%%%%%%%%%%%%%%%%%%%%%%%%%%%%%%%%%%%%%%%%%%%%%%%%%%%%%%%%%%%%%%%%%%%%%%%%%%%%%%%%%%%%%%%%%%%%%%%%%%%%%%%%%%%%%%%%%%%%%%%%%%%%PUNTOS EN K_1 INTERSECTADO CON D
\psdots[dotsize=2pt 0,dotstyle=*](2.5,3)
\psdots[dotsize=2pt 0,dotstyle=*](2.7,3)
\psdots[dotsize=2pt 0,dotstyle=*](2.9,3)
\psdots[dotsize=2pt 0,dotstyle=*](3.1,3)
\psdots[dotsize=2pt 0,dotstyle=*](2.5,3.2)
\psdots[dotsize=2pt 0,dotstyle=*](2.7,3.2)
\psdots[dotsize=2pt 0,dotstyle=*](2.9,3.2)
\psdots[dotsize=2pt 0,dotstyle=*](3.1,3.2)
\psdots[dotsize=2pt 0,dotstyle=*](2.5,3.4)
\psdots[dotsize=2pt 0,dotstyle=*](2.7,3.4)
\psdots[dotsize=2pt 0,dotstyle=*](2.9,3.4)
\psdots[dotsize=2pt 0,dotstyle=*](2.7,3.6)
\psdots[dotsize=2pt 0,dotstyle=*](2.9,3.6)
\psdots[dotsize=2pt 0,dotstyle=*](2.6,3)
\psdots[dotsize=2pt 0,dotstyle=*](2.5,3.1)
\psdots[dotsize=2pt 0,dotstyle=*](2.7,3.1)
\psdots[dotsize=2pt 0,dotstyle=*](2.6,3.2)
\psdots[dotsize=2pt 0,dotstyle=*](3,3)
\psdots[dotsize=2pt 0,dotstyle=*](2.8,3.2)
\psdots[dotsize=2pt 0,dotstyle=*](2.8,3.1)
\psdots[dotsize=2pt 0,dotstyle=*](2.5,3.4)
\psdots[dotsize=2pt 0,dotstyle=*](2.9,3.1)
\psdots[dotsize=2pt 0,dotstyle=*](3,3.2)
\psdots[dotsize=2pt 0,dotstyle=*](3.1,3.1)
\psdots[dotsize=2pt 0,dotstyle=*](3.2,3.1)
\psdots[dotsize=2pt 0,dotstyle=*](3.2,3.2)
\psdots[dotsize=2pt 0,dotstyle=*](3.3,3.1)
\psdots[dotsize=2pt 0,dotstyle=*](3.3,3)
\psdots[dotsize=2pt 0,dotstyle=*](2.4,3.1)
\psdots[dotsize=2pt 0,dotstyle=*](2.4,3.2)
\psdots[dotsize=2pt 0,dotstyle=*](3.4,3.1)
\psdots[dotsize=2pt 0,dotstyle=*](2.6,3.3)
\psdots[dotsize=2pt 0,dotstyle=*](2.7,3.3)
\psdots[dotsize=2pt 0,dotstyle=*](2.6,3.4)
\psdots[dotsize=2pt 0,dotstyle=*](2.8,3.4)
\psdots[dotsize=2pt 0,dotstyle=*](2.9,3.3)
\psdots[dotsize=2pt 0,dotstyle=*](3,3.3)
\psdots[dotsize=2pt 0,dotstyle=*](3,3.4)
\psdots[dotsize=2pt 0,dotstyle=*](3.1,3.3)
\psdots[dotsize=2pt 0,dotstyle=*](2.7,3.5)
\psdots[dotsize=2pt 0,dotstyle=*](2.8,3.5)
\psdots[dotsize=2pt 0,dotstyle=*](2.9,3.5)
\psdots[dotsize=2pt 0,dotstyle=*](3.1,3.5)
\psdots[dotsize=2pt 0,dotstyle=*](2.8,3.7)
\psdots[dotsize=2pt 0,dotstyle=*](2.7,3.7)
\psdots[dotsize=2pt 0,dotstyle=*](2.9,3.7)
\psdots[dotsize=2pt 0,dotstyle=*](2.5,3.7)
\psdots[dotsize=2pt 0,dotstyle=*](2.8,2.9)
\psdots[dotsize=2pt 0,dotstyle=*](2.9,2.9)
\psdots[dotsize=2pt 0,dotstyle=*](2.7,2.9)
\psdots[dotsize=2pt 0,dotstyle=*](2.8,2.7)

%%%%%%%%%%%%%%%%%%%%%%%%%%%%%%%%%%%%%%%%%%%%%%%%%%%%%%%%%%%%%%%%%%%%%%%%%%%%%%%%%%%%%%%%%%%%%%%%%%%%%%%%%%%%%%%%%%%%%%%%%%%%%%%%%%%%%%%%%%%%%%%%%%%%%%PUNTOS EN K2 INTERSECTADO CON D
\psdots[dotsize=2pt 0,dotstyle=*](2.2,1.4)
\psdots[dotsize=2pt 0,dotstyle=*](2.3,1.4)
\psdots[dotsize=2pt 0,dotstyle=*](2.4,1.4)
\psdots[dotsize=2pt 0,dotstyle=*](2.6,1.4)
\psdots[dotsize=2pt 0,dotstyle=*](2.2,1.5)
\psdots[dotsize=2pt 0,dotstyle=*](2.3,1.5)
\psdots[dotsize=2pt 0,dotstyle=*](2.4,1.5)
\psdots[dotsize=2pt 0,dotstyle=*](2.2,1.6)
\psdots[dotsize=2pt 0,dotstyle=*](2.3,1.6)
\psdots[dotsize=2pt 0,dotstyle=*](2.4,1.6)
\psdots[dotsize=2pt 0,dotstyle=*](2.3,1.8)
%%%%%%%%%%%%%%%%%%%%%%%%%%%%%%%%%%%%%%%%%%%%%%%%%%%%%%%%%%%%%%%%%%%%%%%%%%%%%%%%%%%%%%%%%%%%%%%%%%%%%%%%%%%%%%%%%%%%%%%%%%%%%%%%%%%%%%%%%%%%%%%%%%%%%%%%%%PUNTOS EN K3 INTERSECTADO CON D
\psdots[dotsize=2pt 0,dotstyle=*](4,1.3)
\psdots[dotsize=2pt 0,dotstyle=*](4.1,1.3)
\psdots[dotsize=2pt 0,dotstyle=*](4.2,1.3)
\psdots[dotsize=2pt 0,dotstyle=*](4.3,1.3)
\psdots[dotsize=2pt 0,dotstyle=*](4.4,1.3)
\psdots[dotsize=2pt 0,dotstyle=*](4.5,1.3)
\psdots[dotsize=2pt 0,dotstyle=*](4.6,1.3)
\psdots[dotsize=2pt 0,dotstyle=*](4.7,1.3)
\psdots[dotsize=2pt 0,dotstyle=*](4.1,1.4)
\psdots[dotsize=2pt 0,dotstyle=*](4.3,1.4)
\psdots[dotsize=2pt 0,dotstyle=*](4.5,1.4)
\psdots[dotsize=2pt 0,dotstyle=*](4.6,1.4)
\psdots[dotsize=2pt 0,dotstyle=*](4.7,1.4)
\psdots[dotsize=2pt 0,dotstyle=*](4,1.5)
\psdots[dotsize=2pt 0,dotstyle=*](4.1,1.5)
\psdots[dotsize=2pt 0,dotstyle=*](4.2,1.5)
\psdots[dotsize=2pt 0,dotstyle=*](4.3,1.5)
\psdots[dotsize=2pt 0,dotstyle=*](4.4,1.5)
\psdots[dotsize=2pt 0,dotstyle=*](4.5,1.5)
\psdots[dotsize=2pt 0,dotstyle=*](4.6,1.5)
\psdots[dotsize=2pt 0,dotstyle=*](4.7,1.5)
\psdots[dotsize=2pt 0,dotstyle=*](4.8,1.5)
\psdots[dotsize=2pt 0,dotstyle=*](4,1.6)
\psdots[dotsize=2pt 0,dotstyle=*](4.1,1.6)
\psdots[dotsize=2pt 0,dotstyle=*](4.2,1.6)
\psdots[dotsize=2pt 0,dotstyle=*](4.4,1.6)
\psdots[dotsize=2pt 0,dotstyle=*](4.5,1.6)
\psdots[dotsize=2pt 0,dotstyle=*](4.6,1.6)
\psdots[dotsize=2pt 0,dotstyle=*](4.7,1.6)
\psdots[dotsize=2pt 0,dotstyle=*](4,1.7)
\psdots[dotsize=2pt 0,dotstyle=*](4.2,1.7)
\psdots[dotsize=2pt 0,dotstyle=*](4.3,1.7)
\psdots[dotsize=2pt 0,dotstyle=*](4.4,1.7)
\psdots[dotsize=2pt 0,dotstyle=*](4.6,1.7)
\psdots[dotsize=2pt 0,dotstyle=*](4.7,1.7)
\psdots[dotsize=2pt 0,dotstyle=*](4,1.8)
\psdots[dotsize=2pt 0,dotstyle=*](4.1,1.8)
\psdots[dotsize=2pt 0,dotstyle=*](4.2,1.8)
\psdots[dotsize=2pt 0,dotstyle=*](4.3,1.8)
\psdots[dotsize=2pt 0,dotstyle=*](4.4,1.8)
\psdots[dotsize=2pt 0,dotstyle=*](4.5,1.8)
\psdots[dotsize=2pt 0,dotstyle=*](4.6,1.8)
\psdots[dotsize=2pt 0,dotstyle=*](4.7,1.8)
\psdots[dotsize=2pt 0,dotstyle=*](4,1.9)
\psdots[dotsize=2pt 0,dotstyle=*](4.1,1.9)
\psdots[dotsize=2pt 0,dotstyle=*](4.2,1.9)
\psdots[dotsize=2pt 0,dotstyle=*](4.4,1.9)
\psdots[dotsize=2pt 0,dotstyle=*](4.5,1.9)
\psdots[dotsize=2pt 0,dotstyle=*](4.6,1.9)
\psdots[dotsize=2pt 0,dotstyle=*](4.7,1.9)
\psdots[dotsize=2pt 0,dotstyle=*](4,2)
\psdots[dotsize=2pt 0,dotstyle=*](4.2,2)
\psdots[dotsize=2pt 0,dotstyle=*](4.3,2)
\psdots[dotsize=2pt 0,dotstyle=*](4.4,2)
\psdots[dotsize=2pt 0,dotstyle=*](4.5,2)
\psdots[dotsize=2pt 0,dotstyle=*](4,2.1)
\psdots[dotsize=2pt 0,dotstyle=*](4.1,2.1)
\psdots[dotsize=2pt 0,dotstyle=*](4.2,2.1)
\psdots[dotsize=2pt 0,dotstyle=*](4.3,2.1)
\psdots[dotsize=2pt 0,dotstyle=*](4.4,2.1)
\psdots[dotsize=2pt 0,dotstyle=*](4.5,2.1)
\psdots[dotsize=2pt 0,dotstyle=*](4.6,2.1)
\psdots[dotsize=2pt 0,dotstyle=*](4.1,2.2)
\psdots[dotsize=2pt 0,dotstyle=*](4.2,2.2)
\psdots[dotsize=2pt 0,dotstyle=*](4.3,2.2)
\psdots[dotsize=2pt 0,dotstyle=*](4.5,2.2)
\psdots[dotsize=2pt 0,dotstyle=*](3.9,1.3)
\psdots[dotsize=2pt 0,dotstyle=*](3.9,1.4)
\psdots[dotsize=2pt 0,dotstyle=*](3.9,1.5)
\psdots[dotsize=2pt 0,dotstyle=*](3.9,1.6)
\psdots[dotsize=2pt 0,dotstyle=*](3.9,1.7)
\psdots[dotsize=2pt 0,dotstyle=*](3.9,1.8)
\psdots[dotsize=2pt 0,dotstyle=*](3.9,1.9)
\psdots[dotsize=2pt 0,dotstyle=*](3.9,2)
\psdots[dotsize=2pt 0,dotstyle=*](3.9,2.1)
\psdots[dotsize=2pt 0,dotstyle=*](3.8,1.4)
\psdots[dotsize=2pt 0,dotstyle=*](3.8,1.5)
\psdots[dotsize=2pt 0,dotstyle=*](3.8,1.7)
\psdots[dotsize=2pt 0,dotstyle=*](3.8,1.8)
\psdots[dotsize=2pt 0,dotstyle=*](3.8,1.9)
\psdots[dotsize=2pt 0,dotstyle=*](3.8,2)
\psdots[dotsize=2pt 0,dotstyle=*](3.8,2.1)
\psdots[dotsize=2pt 0,dotstyle=*](3.8,2.2)
\psdots[dotsize=2pt 0,dotstyle=*](3.7,1.4)
\psdots[dotsize=2pt 0,dotstyle=*](3.7,1.5)
\psdots[dotsize=2pt 0,dotstyle=*](3.7,1.6)
\psdots[dotsize=2pt 0,dotstyle=*](3.7,1.7)
\psdots[dotsize=2pt 0,dotstyle=*](3.7,1.9)
\psdots[dotsize=2pt 0,dotstyle=*](3.6,1.5)
\psdots[dotsize=2pt 0,dotstyle=*](3.6,1.6)
\psdots[dotsize=2pt 0,dotstyle=*](3.6,1.7)
\psdots[dotsize=2pt 0,dotstyle=*](3.6,1.8)
\psdots[dotsize=2pt 0,dotstyle=*](4.5,1.1)
\psdots[dotsize=2pt 0,dotstyle=*](4.8,1.9)
\psdots[dotsize=2pt 0,dotstyle=*](3.5,1.4)
%%%%%%%%%%%%%%%%%%%%%%%%%%%%%%%%%%%%%%%%%%%%%%%%%%%%%%%%%%%%%%%%%%%%%%%%%%%%%%%%%%%%%%%%%%%%%%%%%%%%%%%%%%%%%%%%%%%%%%%%%%%%%%%%%%%%%%%%%%%%%%%%%%
\rput[tl](4.5,3.8){$I^2$}
\rput[tl](12.2,2){$B(\mathcal{Q},\varepsilon)$}
\psline[linewidth=1pt]{->}(11,1.2)(11,0.5)
\rput[tl](10.8,0.5){$\mathcal{Q}$}

%%%%%%%%%%%%%%%%%%%%%%%%%%%%%%%%%%%%%%%%%%%%%%%%%%%%%%%%%%%%%%%%%%%%%%%%%%%%%%%%%%%%%%%%%%%%%%%%%%%%%%%%%%%%%%%%%%%%%%%%%%%%%%%%%%%%%%%%%%%%%%%%%%%
\psdots[dotsize=2pt 0,dotstyle=*](10.2,1.1)
\psdots[dotsize=2pt 0,dotstyle=*](11.1,1.1)
\psdots[dotsize=2pt 0,dotstyle=*](11.2,1.1)
\psdots[dotsize=2pt 0,dotstyle=*](11.3,1.1)

\psdots[dotsize=2pt 0,dotstyle=*](10.1,1.2)
\psdots[dotsize=2pt 0,dotstyle=*](10.2,1.2)
\psdots[dotsize=2pt 0,dotstyle=*](10.3,1.2)
\psdots[dotsize=2pt 0,dotstyle=*](10.4,1.2)
\psdots[dotsize=2pt 0,dotstyle=*](10.5,1.2)
\psdots[dotsize=2pt 0,dotstyle=*](10.6,1.2)
\psdots[dotsize=2pt 0,dotstyle=*](10.7,1.2)
\psdots[dotsize=2pt 0,dotstyle=*](10.8,1.2)
\psdots[dotsize=2pt 0,dotstyle=*](10.9,1.2)
\psdots[dotsize=2pt 0,dotstyle=*](11,1.2)
\psdots[dotsize=2pt 0,dotstyle=*](11.1,1.2)
\psdots[dotsize=2pt 0,dotstyle=*](11.2,1.2)
\psdots[dotsize=2pt 0,dotstyle=*](11.3,1.2)
\psdots[dotsize=2pt 0,dotstyle=*](11.4,1.2)
\psdots[dotsize=2pt 0,dotstyle=*](11.5,1.2)
\psdots[dotsize=2pt 0,dotstyle=*](11.6,1.2)
\psdots[dotsize=2pt 0,dotstyle=*](11.7,1.2)
\psdots[dotsize=2pt 0,dotstyle=*](11.8,1.2)
\psdots[dotsize=2pt 0,dotstyle=*](11.9,1.2)

\psdots[dotsize=2pt 0,dotstyle=*](10.2,1.3)
\psdots[dotsize=2pt 0,dotstyle=*](10.3,1.3)
\psdots[dotsize=2pt 0,dotstyle=*](10.4,1.3)
\psdots[dotsize=2pt 0,dotstyle=*](10.5,1.3)
\psdots[dotsize=2pt 0,dotstyle=*](10.6,1.3)
\psdots[dotsize=2pt 0,dotstyle=*](10.8,1.3)
\psdots[dotsize=2pt 0,dotstyle=*](10.9,1.3)
\psdots[dotsize=2pt 0,dotstyle=*](11,1.3)
\psdots[dotsize=2pt 0,dotstyle=*](11.1,1.3)
\psdots[dotsize=2pt 0,dotstyle=*](11.2,1.3)
\psdots[dotsize=2pt 0,dotstyle=*](11.3,1.3)
\psdots[dotsize=2pt 0,dotstyle=*](11.4,1.3)
\psdots[dotsize=2pt 0,dotstyle=*](11.5,1.3)
\psdots[dotsize=2pt 0,dotstyle=*](11.6,1.3)
\psdots[dotsize=2pt 0,dotstyle=*](11.7,1.3)
\psdots[dotsize=2pt 0,dotstyle=*](11.8,1.3)

\psdots[dotsize=2pt 0,dotstyle=*](10.1,1.4)
\psdots[dotsize=2pt 0,dotstyle=*](10.2,1.4)
\psdots[dotsize=2pt 0,dotstyle=*](10.3,1.4)
\psdots[dotsize=2pt 0,dotstyle=*](10.4,1.4)
\psdots[dotsize=2pt 0,dotstyle=*](10.5,1.4)
\psdots[dotsize=2pt 0,dotstyle=*](10.6,1.4)
\psdots[dotsize=2pt 0,dotstyle=*](10.7,1.4)
\psdots[dotsize=2pt 0,dotstyle=*](10.8,1.4)
\psdots[dotsize=2pt 0,dotstyle=*](11.1,1.4)
\psdots[dotsize=2pt 0,dotstyle=*](11.2,1.4)
\psdots[dotsize=2pt 0,dotstyle=*](11.3,1.4)
\psdots[dotsize=2pt 0,dotstyle=*](11.4,1.4)
\psdots[dotsize=2pt 0,dotstyle=*](11.5,1.4)
\psdots[dotsize=2pt 0,dotstyle=*](11.6,1.4)
\psdots[dotsize=2pt 0,dotstyle=*](11.7,1.4)
\psdots[dotsize=2pt 0,dotstyle=*](11.8,1.4)

\psdots[dotsize=2pt 0,dotstyle=*](10.3,1.5)
\psdots[dotsize=2pt 0,dotstyle=*](10.4,1.5)
\psdots[dotsize=2pt 0,dotstyle=*](10.5,1.5)
\psdots[dotsize=2pt 0,dotstyle=*](10.6,1.5)
\psdots[dotsize=2pt 0,dotstyle=*](10.7,1.5)
\psdots[dotsize=2pt 0,dotstyle=*](10.8,1.5)
\psdots[dotsize=2pt 0,dotstyle=*](10.9,1.5)
\psdots[dotsize=2pt 0,dotstyle=*](11.1,1.5)
\psdots[dotsize=2pt 0,dotstyle=*](11.2,1.5)
\psdots[dotsize=2pt 0,dotstyle=*](11.3,1.5)
\psdots[dotsize=2pt 0,dotstyle=*](11.4,1.5)
\psdots[dotsize=2pt 0,dotstyle=*](11.5,1.5)
\psdots[dotsize=2pt 0,dotstyle=*](11.6,1.5)
\psdots[dotsize=2pt 0,dotstyle=*](11.8,1.5)

\psdots[dotsize=2pt 0,dotstyle=*](10.2,1.6)
\psdots[dotsize=2pt 0,dotstyle=*](10.5,1.6)
\psdots[dotsize=2pt 0,dotstyle=*](10.6,1.6)
\psdots[dotsize=2pt 0,dotstyle=*](10.7,1.6)
\psdots[dotsize=2pt 0,dotstyle=*](10.8,1.6)
\psdots[dotsize=2pt 0,dotstyle=*](10.9,1.6)
\psdots[dotsize=2pt 0,dotstyle=*](11,1.6)
\psdots[dotsize=2pt 0,dotstyle=*](11.1,1.6)
\psdots[dotsize=2pt 0,dotstyle=*](11.2,1.6)
\psdots[dotsize=2pt 0,dotstyle=*](11.3,1.6)
\psdots[dotsize=2pt 0,dotstyle=*](11.4,1.6)
\psdots[dotsize=2pt 0,dotstyle=*](11.5,1.6)
\psdots[dotsize=2pt 0,dotstyle=*](11.8,1.6)
\psdots[dotsize=2pt 0,dotstyle=*](11.9,1.6)

\psdots[dotsize=2pt 0,dotstyle=*](10.2,1.7)
\psdots[dotsize=2pt 0,dotstyle=*](10.3,1.7)
\psdots[dotsize=2pt 0,dotstyle=*](10.4,1.7)
\psdots[dotsize=2pt 0,dotstyle=*](10.5,1.7)
\psdots[dotsize=2pt 0,dotstyle=*](10.6,1.7)
\psdots[dotsize=2pt 0,dotstyle=*](10.7,1.7)
\psdots[dotsize=2pt 0,dotstyle=*](10.8,1.7)
\psdots[dotsize=2pt 0,dotstyle=*](10.9,1.7)
\psdots[dotsize=2pt 0,dotstyle=*](11,1.7)
\psdots[dotsize=2pt 0,dotstyle=*](11.1,1.7)
\psdots[dotsize=2pt 0,dotstyle=*](11.2,1.7)
\psdots[dotsize=2pt 0,dotstyle=*](11.3,1.7)
\psdots[dotsize=2pt 0,dotstyle=*](11.4,1.7)
\psdots[dotsize=2pt 0,dotstyle=*](11.5,1.7)
\psdots[dotsize=2pt 0,dotstyle=*](11.6,1.7)
\psdots[dotsize=2pt 0,dotstyle=*](11.7,1.7)
\psdots[dotsize=2pt 0,dotstyle=*](11.8,1.7)

\psdots[dotsize=2pt 0,dotstyle=*](10.2,1.8)
\psdots[dotsize=2pt 0,dotstyle=*](10.3,1.8)
\psdots[dotsize=2pt 0,dotstyle=*](10.4,1.8)
\psdots[dotsize=2pt 0,dotstyle=*](10.5,1.8)
\psdots[dotsize=2pt 0,dotstyle=*](10.6,1.8)
\psdots[dotsize=2pt 0,dotstyle=*](10.7,1.8)
\psdots[dotsize=2pt 0,dotstyle=*](10.9,1.8)
\psdots[dotsize=2pt 0,dotstyle=*](11,1.8)
\psdots[dotsize=2pt 0,dotstyle=*](11.1,1.8)
\psdots[dotsize=2pt 0,dotstyle=*](11.2,1.8)
\psdots[dotsize=2pt 0,dotstyle=*](11.3,1.8)
\psdots[dotsize=2pt 0,dotstyle=*](11.4,1.8)
\psdots[dotsize=2pt 0,dotstyle=*](11.5,1.8)
\psdots[dotsize=2pt 0,dotstyle=*](11.7,1.8)
\psdots[dotsize=2pt 0,dotstyle=*](11.8,1.8)

\psdots[dotsize=2pt 0,dotstyle=*](10.1,1.9)
\psdots[dotsize=2pt 0,dotstyle=*](10.2,1.9)
\psdots[dotsize=2pt 0,dotstyle=*](10.3,1.9)
\psdots[dotsize=2pt 0,dotstyle=*](10.5,1.9)
\psdots[dotsize=2pt 0,dotstyle=*](10.6,1.9)
\psdots[dotsize=2pt 0,dotstyle=*](10.8,1.9)
\psdots[dotsize=2pt 0,dotstyle=*](10.9,1.9)
\psdots[dotsize=2pt 0,dotstyle=*](11,1.9)
\psdots[dotsize=2pt 0,dotstyle=*](11.1,1.9)
\psdots[dotsize=2pt 0,dotstyle=*](11.2,1.9)
\psdots[dotsize=2pt 0,dotstyle=*](11.3,1.9)
\psdots[dotsize=2pt 0,dotstyle=*](11.4,1.9)
\psdots[dotsize=2pt 0,dotstyle=*](11.6,1.9)
\psdots[dotsize=2pt 0,dotstyle=*](11.7,1.9)
\psdots[dotsize=2pt 0,dotstyle=*](11.8,1.9)

\psdots[dotsize=2pt 0,dotstyle=*](10.1,2)
\psdots[dotsize=2pt 0,dotstyle=*](10.2,2)
\psdots[dotsize=2pt 0,dotstyle=*](10.3,2)
\psdots[dotsize=2pt 0,dotstyle=*](10.4,2)
\psdots[dotsize=2pt 0,dotstyle=*](10.5,2)
\psdots[dotsize=2pt 0,dotstyle=*](10.6,2)
\psdots[dotsize=2pt 0,dotstyle=*](10.7,2)
\psdots[dotsize=2pt 0,dotstyle=*](10.8,2)
\psdots[dotsize=2pt 0,dotstyle=*](10.9,2)
\psdots[dotsize=2pt 0,dotstyle=*](11,2)
\psdots[dotsize=2pt 0,dotstyle=*](11.1,2)
\psdots[dotsize=2pt 0,dotstyle=*](11.2,2)
\psdots[dotsize=2pt 0,dotstyle=*](11.3,2)
\psdots[dotsize=2pt 0,dotstyle=*](11.4,2)
\psdots[dotsize=2pt 0,dotstyle=*](11.5,2)
\psdots[dotsize=2pt 0,dotstyle=*](11.6,2)
\psdots[dotsize=2pt 0,dotstyle=*](11.7,2)
\psdots[dotsize=2pt 0,dotstyle=*](11.8,2)
\psdots[dotsize=2pt 0,dotstyle=*](11.9,2)

\psdots[dotsize=2pt 0,dotstyle=*](10.2,2.1)
\psdots[dotsize=2pt 0,dotstyle=*](10.4,2.1)
\psdots[dotsize=2pt 0,dotstyle=*](10.5,2.1)
\psdots[dotsize=2pt 0,dotstyle=*](10.6,2.1)
\psdots[dotsize=2pt 0,dotstyle=*](10.8,2.1)
\psdots[dotsize=2pt 0,dotstyle=*](10.9,2.1)
\psdots[dotsize=2pt 0,dotstyle=*](11,2.1)
\psdots[dotsize=2pt 0,dotstyle=*](11.1,2.1)
\psdots[dotsize=2pt 0,dotstyle=*](11.2,2.1)
\psdots[dotsize=2pt 0,dotstyle=*](11.3,2.1)
\psdots[dotsize=2pt 0,dotstyle=*](11.5,2.1)
\psdots[dotsize=2pt 0,dotstyle=*](11.6,2.1)
\psdots[dotsize=2pt 0,dotstyle=*](11.7,2.1)
\psdots[dotsize=2pt 0,dotstyle=*](11.8,2.1)

\psdots[dotsize=2pt 0,dotstyle=*](10.2,2.2)
\psdots[dotsize=2pt 0,dotstyle=*](10.3,2.2)
\psdots[dotsize=2pt 0,dotstyle=*](10.4,2.2)
\psdots[dotsize=2pt 0,dotstyle=*](10.5,2.2)
\psdots[dotsize=2pt 0,dotstyle=*](10.6,2.2)
\psdots[dotsize=2pt 0,dotstyle=*](10.7,2.2)
\psdots[dotsize=2pt 0,dotstyle=*](11,2.2)
\psdots[dotsize=2pt 0,dotstyle=*](11.1,2.2)
\psdots[dotsize=2pt 0,dotstyle=*](11.2,2.2)
\psdots[dotsize=2pt 0,dotstyle=*](11.3,2.2)
\psdots[dotsize=2pt 0,dotstyle=*](11.4,2.2)
\psdots[dotsize=2pt 0,dotstyle=*](11.5,2.2)
\psdots[dotsize=2pt 0,dotstyle=*](11.6,2.2)
\psdots[dotsize=2pt 0,dotstyle=*](11.7,2.2)
\psdots[dotsize=2pt 0,dotstyle=*](11.8,2.2)

\psdots[dotsize=2pt 0,dotstyle=*](10.2,2.3)
\psdots[dotsize=2pt 0,dotstyle=*](10.3,2.3)
\psdots[dotsize=2pt 0,dotstyle=*](10.4,2.3)
\psdots[dotsize=2pt 0,dotstyle=*](10.5,2.3)
\psdots[dotsize=2pt 0,dotstyle=*](10.7,2.3)
\psdots[dotsize=2pt 0,dotstyle=*](10.8,2.3)
\psdots[dotsize=2pt 0,dotstyle=*](10.9,2.3)
\psdots[dotsize=2pt 0,dotstyle=*](11,2.3)
\psdots[dotsize=2pt 0,dotstyle=*](11.1,2.3)
\psdots[dotsize=2pt 0,dotstyle=*](11.2,2.3)
\psdots[dotsize=2pt 0,dotstyle=*](11.3,2.3)
\psdots[dotsize=2pt 0,dotstyle=*](11.4,2.3)
\psdots[dotsize=2pt 0,dotstyle=*](11.5,2.3)
\psdots[dotsize=2pt 0,dotstyle=*](11.6,2.3)
\psdots[dotsize=2pt 0,dotstyle=*](11.9,2.3)

\psdots[dotsize=2pt 0,dotstyle=*](10.2,2.4)
\psdots[dotsize=2pt 0,dotstyle=*](10.3,2.4)
\psdots[dotsize=2pt 0,dotstyle=*](10.4,2.4)
\psdots[dotsize=2pt 0,dotstyle=*](10.5,2.4)
\psdots[dotsize=2pt 0,dotstyle=*](10.6,2.4)
\psdots[dotsize=2pt 0,dotstyle=*](10.7,2.4)
\psdots[dotsize=2pt 0,dotstyle=*](10.9,2.4)
\psdots[dotsize=2pt 0,dotstyle=*](11,2.4)
\psdots[dotsize=2pt 0,dotstyle=*](11.1,2.4)
\psdots[dotsize=2pt 0,dotstyle=*](11.2,2.4)
\psdots[dotsize=2pt 0,dotstyle=*](11.3,2.4)
\psdots[dotsize=2pt 0,dotstyle=*](11.4,2.4)
\psdots[dotsize=2pt 0,dotstyle=*](11.6,2.4)
\psdots[dotsize=2pt 0,dotstyle=*](11.7,2.4)
\psdots[dotsize=2pt 0,dotstyle=*](11.8,2.4)

\psdots[dotsize=2pt 0,dotstyle=*](10.2,2.5)
\psdots[dotsize=2pt 0,dotstyle=*](10.3,2.5)
\psdots[dotsize=2pt 0,dotstyle=*](10.4,2.5)
\psdots[dotsize=2pt 0,dotstyle=*](10.5,2.5)
\psdots[dotsize=2pt 0,dotstyle=*](10.6,2.5)
\psdots[dotsize=2pt 0,dotstyle=*](10.7,2.5)
\psdots[dotsize=2pt 0,dotstyle=*](10.8,2.5)
\psdots[dotsize=2pt 0,dotstyle=*](10.9,2.5)
\psdots[dotsize=2pt 0,dotstyle=*](11,2.5)
\psdots[dotsize=2pt 0,dotstyle=*](11.1,2.5)
\psdots[dotsize=2pt 0,dotstyle=*](11.2,2.5)
\psdots[dotsize=2pt 0,dotstyle=*](11.3,2.5)
\psdots[dotsize=2pt 0,dotstyle=*](11.4,2.5)
\psdots[dotsize=2pt 0,dotstyle=*](11.7,2.5)
\psdots[dotsize=2pt 0,dotstyle=*](11.8,2.5)

\psdots[dotsize=2pt 0,dotstyle=*](10.1,2.6)
\psdots[dotsize=2pt 0,dotstyle=*](10.2,2.6)
\psdots[dotsize=2pt 0,dotstyle=*](10.4,2.6)
\psdots[dotsize=2pt 0,dotstyle=*](10.5,2.6)
\psdots[dotsize=2pt 0,dotstyle=*](10.6,2.6)
\psdots[dotsize=2pt 0,dotstyle=*](10.7,2.6)
\psdots[dotsize=2pt 0,dotstyle=*](10.8,2.6)
\psdots[dotsize=2pt 0,dotstyle=*](10.9,2.6)
\psdots[dotsize=2pt 0,dotstyle=*](11,2.6)
\psdots[dotsize=2pt 0,dotstyle=*](11.1,2.6)
\psdots[dotsize=2pt 0,dotstyle=*](11.2,2.6)
\psdots[dotsize=2pt 0,dotstyle=*](11.3,2.6)
\psdots[dotsize=2pt 0,dotstyle=*](11.4,2.6)
\psdots[dotsize=2pt 0,dotstyle=*](11.5,2.6)
\psdots[dotsize=2pt 0,dotstyle=*](11.7,2.6)
\psdots[dotsize=2pt 0,dotstyle=*](11.8,2.6)

\psdots[dotsize=2pt 0,dotstyle=*](10.2,2.7)
\psdots[dotsize=2pt 0,dotstyle=*](10.3,2.7)
\psdots[dotsize=2pt 0,dotstyle=*](10.4,2.7)
\psdots[dotsize=2pt 0,dotstyle=*](10.5,2.7)
\psdots[dotsize=2pt 0,dotstyle=*](10.6,2.7)
\psdots[dotsize=2pt 0,dotstyle=*](10.7,2.7)
\psdots[dotsize=2pt 0,dotstyle=*](10.9,2.7)
\psdots[dotsize=2pt 0,dotstyle=*](11,2.7)
\psdots[dotsize=2pt 0,dotstyle=*](11.1,2.7)
\psdots[dotsize=2pt 0,dotstyle=*](11.2,2.7)
\psdots[dotsize=2pt 0,dotstyle=*](11.3,2.7)
\psdots[dotsize=2pt 0,dotstyle=*](11.4,2.7)
\psdots[dotsize=2pt 0,dotstyle=*](11.5,2.7)
\psdots[dotsize=2pt 0,dotstyle=*](11.6,2.7)
\psdots[dotsize=2pt 0,dotstyle=*](11.7,2.7)
\psdots[dotsize=2pt 0,dotstyle=*](11.8,2.7)

\psdots[dotsize=2pt 0,dotstyle=*](10.2,2.8)
\psdots[dotsize=2pt 0,dotstyle=*](10.3,2.8)
\psdots[dotsize=2pt 0,dotstyle=*](10.4,2.8)
\psdots[dotsize=2pt 0,dotstyle=*](10.5,2.8)
\psdots[dotsize=2pt 0,dotstyle=*](10.6,2.8)
\psdots[dotsize=2pt 0,dotstyle=*](10.7,2.8)
\psdots[dotsize=2pt 0,dotstyle=*](10.8,2.8)
\psdots[dotsize=2pt 0,dotstyle=*](10.9,2.8)
\psdots[dotsize=2pt 0,dotstyle=*](11,2.8)
\psdots[dotsize=2pt 0,dotstyle=*](11.1,2.8)
\psdots[dotsize=2pt 0,dotstyle=*](11.2,2.8)
\psdots[dotsize=2pt 0,dotstyle=*](11.3,2.8)
\psdots[dotsize=2pt 0,dotstyle=*](11.4,2.8)
\psdots[dotsize=2pt 0,dotstyle=*](11.5,2.8)
\psdots[dotsize=2pt 0,dotstyle=*](11.7,2.8)
\psdots[dotsize=2pt 0,dotstyle=*](11.8,2.8)

\psdots[dotsize=2pt 0,dotstyle=*](10.1,2.9)
\psdots[dotsize=2pt 0,dotstyle=*](10.7,2.9)
\psdots[dotsize=2pt 0,dotstyle=*](10.8,2.9)
\psdots[dotsize=2pt 0,dotstyle=*](11.8,2.9)
\end{scriptsize}
\end{pspicture}
\caption{The set $\mathcal{H}\cap\mathcal{R}_n$ is mapped under $f_n$ into $B(\mathcal{Q},\varepsilon)$ and $f_n^{-1}(\mathcal{Q}\cap (1/l_n)\mathbb{Z}^2)$ is contained in the $\varepsilon$-neighbourhood of $\mathcal{H}$.}
\end{figure}

\begin{lemma}\label{conv lemma 2}
Given $0<\varepsilon<\min\{d(\mathcal{H},\partial F^{-1}(W)); d(\mathcal{Q},\partial W); \displaystyle\min_{i<j} d(\mathcal{H}_i,\mathcal{H}_j)\}$, there is a positive integer $n_0=n_0(\varepsilon)$ such that for every $n\geq n_0$ the following holds:

\begin{itemize}
\item[i) ]$f_n(\mathcal{H}\cap\mathcal{R}_n)$ is contained in $B(\mathcal{Q},\varepsilon)$;

\item[ii) ]$f_n^{-1}\left(\mathcal{Q}\cap(1/l_n)\mathbb{Z}^2\right)\subset B(\mathcal{H},\varepsilon)$.
\end{itemize}
\end{lemma}

\begin{proof}

Assume that i) does not hold. Then there must exist an increasing sequence of integers $k_n$ and a sequence of points $x_{n}\in\mathcal{H}\cap\mathcal{R}_{k_n}$ such that $f_{k_n}(x_{n})$ does not belong to $B(\mathcal{Q},\varepsilon)$. From the compactness of $\mathcal{H}$ and after passing to a subsequence, we may assume that $(x_{n})_{n\geq 1}$ converges to a point $x\in \mathcal{H}$. By the uniform convergence of $\widehat{f}_n$ to $F$ (recall that $\widehat{f}_n$ was defined in (\ref{extension})), we have that $f_{k_n}(x_{n})\longrightarrow F(x)$. Since $B(\mathcal{Q},\varepsilon)$ is an open set, $F(x)$ cannot belong to $B(\mathcal{Q},\varepsilon)$. However, this contradicts the fact that $\mathcal{H}\subset F^{-1} (B(\mathcal{Q},\varepsilon))$.\\

To prove ii) we proceed also by contradiction. Suppose that there exist an increasing sequence of positive integers $(i_n)_{n\in\mathbb{N}}$ and a sequence of points $(u_{i_n})_{n\in\mathbb{N}}$ such that for every $n\geq 1$ we have that $u_{i_n}$ belongs to $f_{i_n}^{-1}(\mathcal{Q}\cap (1/l_{i_n})\mathbb{Z}^2)\setminus B(\mathcal{H},\varepsilon)$. Observe that $(u_{i_n})_{n\in\mathbb{N}}$ converges, up to a subsequence, to an element $u\in I^2\setminus B(\mathcal{H},\varepsilon)$. On the other hand, by the compactness of $\mathcal{Q}$ and since $\widehat{f}_n$ converges uniformly to $F$, we have that $f_{i_n}(u_{i_n})$ converges to $F(u)\in \mathcal{Q}$. However this would imply that $u\in\mathcal{H}$, which is impossible since $u\in I^2\setminus B(\mathcal{H},\varepsilon)$.
\end{proof}

From now on, for each natural number $n$ we consider $f_n$ as a map from $\phi_n(\mathcal{D}_{\rho})$ to $(1/l_n)\mathbb{Z}^2$; then the extension of the restriction of $f_n$ over $\mathcal{R}_n$, namely $\widehat{f_n|_{\mathcal{R}_n}}$, converges uniformly to $F$. Consider $k\in\mathbb{N}$ satisfying that

\begin{equation}\label{k}
0<\frac{1}{k}<\min\left\{\frac{d(\mathcal{H},\partial F^{-1}(W))}{2\mathsf{Reg}(F)^2+1},\ \min_{i<j}d(\mathcal{H}_i,\mathcal{H}_j),\ d(\mathcal{Q},\partial W)\right\}.
\end{equation}

For every non-negative integer $j$, let $\mathcal{Q}_{k,j}$ be the set points $y\in B(\mathcal{Q},1/k)$ for which $d(y,\partial B(\mathcal{Q},1/k))>jL/l_n$; notice that $\mathcal{Q}_{k,j}$ can be empty for a sufficiently large positive integer $j$. The key result to prove Proposition \ref{local non realizable} is given by Lemma \ref{claim4} below, which claims that for sufficiently large positive integers $k,n$, there is an annulus $A_{k,n}$ centred at the origin and with external radius equal to the diameter of $f_n^{-1}(\mathcal{Q}_{k,0})$, such that $f_n^{-1}(\mathcal{Q}\cap(1/l_n)\mathbb{Z}^2)\cap A_{k,n}$ is mapped under $f_n$ into a neighbourhood of $\partial B(\mathcal{Q},1/k)$. Concretely, Lemma \ref{claim4} is a consequence of the fact that points in the pre-image under $f_n$ of a closed ball $B$ contained in $W$ which are distant from the origin, must be mapped by $f_n$ ``close" to the boundary of $B$; this last fact (Lemma \ref{claim1} below) was inspired by the proof of Lemma 6 in \cite{CN} and relies strongly on the geometry of $\mathcal{D}_{\rho}$, namely, on the $2\mathbb{Z}^2$-property. 

\medskip

For every $j\geq 0$ such that $\mathcal{Q}_{k,j}\neq\emptyset$, let $x_{k,j}^{(n)}$ be a point in $f_n^{-1}(\mathcal{Q}_{k,j})$ with the property that $||x_{k,j}^{(n)}||$ is maximal. 
Let $i(n):=\max\{i\in\mathbb{N}:\ \mathcal{Q}_{k,i}\neq\emptyset\}$, i.e, $i(n)$ satisfies that the side-length of $\mathcal{Q}_{k,i(n)}$ is less or equal than $2L/l_n$. We claim that the distance from $f(x_{k,j}^{(n)})$ to $\partial\mathcal{Q}_{k,j}$ cannot be larger than $L/l_n$.

\begin{lemma}\label{claim1}
For every $j<i(n)$ there holds $d(f_n(x_{k,j}^{(n)}), \partial\mathcal{Q}_{k,j})\leq\, L/l_n$.
\end{lemma}

\begin{proof}
Indeed, if $d(f_n(x_{k,j}^{(n)}),\partial\mathcal{Q}_{k,j})> L/l_n$, then $B(f_n(x_{k,j}^{(n)}),L/l_n)$ is strictly contained in $\mathcal{Q}_{k,j}$. On the other hand, the $2\mathbb{Z}^2$-property implies that at least one of the points $x_{k,j}^{(n)}\pm e_{1,n}, x_{k,j}^{(n)}\pm e_{2,n}, x_{k,j}^{(n)}\pm e_{1,n}\pm e_{2,n}$ (where $e_1=(1/l_n,0)$ and $e_2=(0,1/l_n)$) belongs to $\phi_n(\mathcal{D}_{\rho})$ and lies at distance from the origin larger than that $||x_{k,j}^{(n)}||$; see Figure 5 for a picture of this situation when $x_{k,j}^{(n)}$ lies to the left-hand side of the square $B(||x_{k,j}^{(n)}||)$. Then, from the $L$-Lipschitz condition, this point is mapped by $f_n$ into $B(f(x_{k,j}^{(n)}),L/l_n)\subset\mathcal{Q}_{k,j}$, contradicting the choice of $x_{k,j}^{(n)}$.\\

\begin{figure}[h!]
\psset{xunit=0.7cm,yunit=0.7cm,algebraic=true,dimen=middle,dotstyle=o,dotsize=3pt 0,linewidth=1pt,arrowsize=3pt 2,arrowinset=0.25}
\begin{pspicture}(-7,0)(10,8)
\pspolygon[linewidth=1pt](2,1)(8,1)(8,7)(2,7)
\begin{scriptsize}
\psdots[dotstyle=*](2,2)
\psline{->}(1.8,2)(0.8,2)
\rput[bl](0,1.8){$x_{k,j}^{(n)}$}
\psdots[dotstyle=*](1.6,2.4)
\psdots[dotstyle=*](1.6,1.6)
\psdots[dotstyle=*](2,1.6)
\psdots[dotstyle=*](2.4,1.6)
\psdots[dotstyle=*](2,2.4)
\psdots[dotstyle=*](2.4,2.4)
\psdots[dotstyle=*](2.4,2)
\rput[bl](4,7.5){$B(||x_{k,j}^{(n)}||)$}
\end{scriptsize}
\end{pspicture}
\caption{A possible configuration of $\phi_n(\mathcal{D}_{\rho})$ around $x_{k,j}^{(n)}$.}
\end{figure}
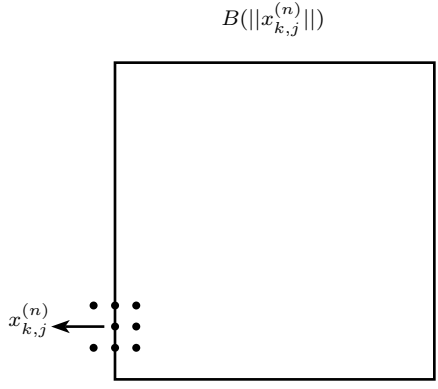

\end{proof}

\begin{lemma}\label{claim2}
For every positive integer $j<i(n)$ there holds $||x_{k,j}^{(n)}||<||x_{k,j-1}^{(n)}||$. In particular the set $\mathsf{Ann}((0,0),||x_{k,j}^{(n)}||,||x_{k,j-1}^{(n)}||)$ is non-empty. 
\end{lemma}

\begin{proof}
We proceed by contradiction. By construction of $x_{k,j}^{(n)}$, it is sufficient to consider the case \begin{equation}\label{disteq}
||x_{k,j}^{(n)}||=||x_{k,j-1}^{(n)}||.
\end{equation}

In this case by Lemma \ref{claim1}, the preceding equality \eqref{disteq} and the definition of $x_{k,j}^{(n)}$, we get that 

\begin{equation*}
d(f_n(x_{k,j-1}^{(n)}),\partial\mathcal{Q}_{k,j-1})\leq L/l_n\qquad\text{and}\qquad d(f_n(x_{k,j-1}^{(n)}),\partial\mathcal{Q}_{k,j})\leq L/l_n
\end{equation*}

This implies that $f_n(x_{k,j-1}^{(n)})\in\partial\mathcal{Q}_{k,j-1}\cap\mathcal{Q}_{k,j-1}$, which is impossible since $\mathcal{Q}_{k,j-1}$ is an open set.
\end{proof}

Now we shall prove that $x_{k,i(n)}^{(n)}$ belongs to $I^2$ for a sufficiently large $n\in\mathbb{N}$.

\begin{lemma}\label{claim3}
Given $k$ as in \eqref{k}, there must exist a positive integer $n_k$ depending on $k$ such that for every $n\geq n_k$ the point $x_{k,i(n)}^{(n)}$ belongs to $\mathcal{H}$.
\end{lemma}

\begin{proof}
By the bi-Lipschitz decomposition of $F$, there is a large-enough $n_k\in\mathbb{N}$ such that for every $n\geq n_k$ there holds

\begin{equation*}
d(\partial F^{-1}(\mathcal{Q}_{k,i(n)}), \partial\mathcal{H})>\frac{1}{k}.
\end{equation*}

Hence the claim follows by applying part ii) of Lemma \ref{conv lemma 2} to $\mathcal{Q}_{k,i(n_k)}$, $\varepsilon=1/k$ and by observing that $\mathcal{Q}_{k,i(n)}\subset\mathcal{Q}_{k,i(n_k)}$ for every $n\geq n_k$. 
\end{proof}

\begin{lemma}\label{claim4}
For every $0\leq j\leq i(n)$, the set $\mathsf{Ann}((0,0),||x_{k,j}^{(n)}||,||x_{k,j-1}^{(n)}||)\cap f_n^{-1}(\mathcal{Q}_{k,0})$ is mapped into $\mathcal{Q}_{k,j-1}\setminus\mathcal{Q}_{k,j}$ under $f_n$.
\end{lemma}

\begin{proof}
From the maximality of $||x_{k,1}^{(n)}||$, we get 

\begin{equation*}
\mathsf{Ann}((0,0),||x_{k,1}^{(n)}||,||x_{k,0}^{(n)}||)\cap f_n^{-1}(\mathcal{Q}_{k,0})\subset\mathcal{Q}_{k,0}\setminus\mathcal{Q}_{k,1}.
\end{equation*}

Otherwise, there must exists $u\in f_n^{-1}(\mathcal{Q}_{k,0})$ with $||u||>||x_{k,1}^{(n)}||$ such that $f_{n}(u)\in \mathcal{Q}_{k,1}$, contradicting the choice of $x_{k,1}^{(n)}\in f_n^{-1}(\mathcal{Q}_{k,1})$. Analogously, for every non-negative integer $j$ such that $\mathcal{Q}_{k,j}\neq\emptyset$ there holds:

\begin{equation*}
\mathsf{Ann}((0,0),||x_{k,j}^{(n)}||,||x_{k,j-1}^{(n)}||)\cap f_n^{-1}(\mathcal{Q}_{k,j-1})\subset\mathcal{Q}_{k,j-1}\setminus\mathcal{Q}_{k,j}.
\end{equation*}
\end{proof}

In particular, from Lemma \ref{claim4} we have that there are no points in $\mathsf{Ann}((0,0),||x_{k,j}^{(n)}||,||x_{k,j-1}^{(n)}||)\cap f_n^{-1}(\mathcal{Q}_{k,j-1})$ which are sent into $\mathcal{Q}_{k,0}\setminus\mathcal{Q}_{k,j-1}$.

\medskip

Henceforth, we consider the points $x_{k,j}:=x_{k,j}^{(n)}$ where the $x_{k,j}^{(n)}$'s are the points as in the previous lemmas. For every $n\in\mathbb{N}$ let $j(n)$ be the minimum positive integer with the property that $x_{k,j}\in I^2$, which in fact must exists as a consequence of Lemma \ref{claim3}.

\begin{remark}\label{remcentpoint}
If $\mathcal{Q}$ is centred at a point $y_m\in (1/l_m)\mathbb{Z}^2$, then from Lemma \ref{claim4} we conclude that for a sufficiently large positive integer $n$, the pre-image of the center of $\mathcal{Q}$ under $f_n$, namely $f_{n}^{-1}(y_m)$, belongs to $I^2$.
\end{remark}
\medskip

Notice that for every $k\in\mathbb{N}$ satisfying \eqref{k}, the point $x_{k, j(n_k)}$ belongs to $B(\mathcal{H}, (2\mathsf{Reg}(F)^2+1)/k)$ (where $n_k$ is the positive integer given in the proof of Lemma \ref{claim3}), and hence the sequence $(x_{k,j(n_k)})_{k\geq 1}$ converges (under a subsequence) to an element $x_{lim}\in\mathcal{H}$. In addition, the sequence of balls $(\mathcal{Q}_{k,j(n_k)})_{k\geq 0}$ converges in the Hausdorff distance (under a subsequence) to a closed ball $\mathcal{M}\subset\mathcal{Q}$. We use this fact to prove the next claim. 
\medskip

\begin{lemma}\label{H=Q}
The set $\mathcal{H}$ is contained in $\overline{B}(||x_{lim}||)$. In particular $x_{lim}$ belongs to $\partial\mathcal{H}$.
\end{lemma}

\begin{proof}
Suppose there is a point $u\in int(\mathcal{H}_i)\setminus\overline{B}(||x_{lim}||)$ for some $1\leq i\leq N$. Thereupon there must exist a positive number $\alpha$ such that $\overline{B}(u,\alpha)$ is contained $int(\mathcal{H}_i)\setminus \overline{B}(||x_{lim}||)$. Thus, from Lemma \ref{claim4} and by the definition of $j(n_k)$, there is a sufficiently large $k_0\in\mathbb{N}$ such that $f_{n_k}(\overline{B}(u,\alpha)\cap\mathcal{R}_{n_k})$ is contained in $\mathcal{Q}_{k,j(n_k)-1}\setminus\mathcal{Q}_{k,j(n_k)}$ for every $k\geq k_0$. Hence, after passing to the limit, we obtain that $F(\overline{B}(u,\alpha))\subset\partial\mathcal{M}$, i.e, $F$ maps a set with positive Lebesgue measure into a set with zero-Lebesgue measure, which is impossible since $F|_{\mathcal{H}_i}$ is a bi-Lipschitz homeomorphism. Thus, $int(\mathcal{H})\subset \overline{B}(||x_{lim}||)$ and since each connected component of $\mathcal{H}$ is homeomorphic to a closed ball, we conclude that $\mathcal{H}\subset \overline{B}(||x_{lim}||)$. 
\end{proof}

\subsection{Proof of Proposition 3.2}\label{Proof3.2}

In order to prove Proposition \ref{local non realizable} we need to estimate the cardinality of the set of points in $\mathcal{Q}\cap (1/l_n)\mathbb{Z}^2$ which have no pre-image under $f_n:\mathcal{R}_n\to (1/l_n)\mathbb{Z}^2$. To treat these points, we require the following result which corresponds to a slightly different version of Lemma 3.1 in \cite{checos2} and whose proof is analogous.

\vspace{0.2cm}

\begin{lemma}\label{sym diff}
Let $U\subset\mathbb{R}^d$ be a closed set which is the image of $I^d$ under a bi-Lipschitz map and let $g:U\to\mathbb{R}^d$ be a homeomorphism and $h:U\to\mathbb{R}^d$ be continuous. Then $h(U)\Delta g(U)$ is contained in $\overline{B}(\partial g(U), ||g-h||_{\infty})$, where $\Delta$ denotes the symmetric difference.
\end{lemma}

\medskip

Before showing Proposition \ref{local non realizable}, we introduce some terminology. For every $n\in\mathbb{N}$ consider the normalized counting measures $\mu_n$ and $\nu_n$ defined by 

\begin{equation*}
\begin{split}
\mu_n(A):=\dfrac{|A\cap\mathcal{R}_n|}{l_n^2},\hspace{1cm}
\nu_n(C):=\dfrac{|C\cap\frac{1}{l_n}\mathbb{Z}^2|}{l_n^2}
\end{split}
\end{equation*}

Notice that $(\nu_n|_{F(I^2)})_{n\in\mathbb{N}}$ converges weakly to the Lebesgue measure in $F(I^2)$. Moreover, it can be shown that $\mu_n$ converges weakly to the measure $\rho\lambda$ in $I^2$; the proof of these facts follows the very same lines as Claims 5.3.1 and 5.3.2 in \cite{checos}.

\medskip

\begin{lemma}\label{conv measures}
The sequences of measures $(\mu_n)_{n\in\mathbb{N}}$ and $(\nu_n)_{n\in\mathbb{N}}$ converge weakly to $\rho\lambda$ and $\lambda$, respectively. In particular, $(\widehat{f}_n)_{\#}(\mu_n)$ converges weakly to $F_{\#}(\rho\lambda)$. 
\end{lemma}

\medskip

Now we are in the position to prove Proposition \ref{local non realizable}. We prove this by showing that the loss of mass in $\mathcal{Q}$ only happens close to its boundary. Recall that $\mathcal{Q}$ is centred at a point $y=y_m\in (1/l_{m})\mathbb{Z}^2$ for some $m\in\mathbb{N}$; since $(l_n)_{n\geq 1}$ satisfy that $l_{n}$ divides to $l_{n+1}$, then $x\in(1/l_n)\mathbb{Z}^2$ for every $n\geq m$.

\medskip

\begin{proof}[Proof of Proposition \ref{local non realizable}]
To verify that $F_{\#}(\rho\lambda)|_{\mathcal{Q}}=\lambda|_{\mathcal{Q}}$, by Lemma \ref{conv measures} it is sufficient to show that the sequence $(\hat{f}_{n_k}|_{F^{-1}(\mathcal{Q})})_{\#}\mu_{n_k}$ converges weakly to $\lambda |_{\mathcal{Q}}$. We proceed in a similar way as in the proof of Lemma 3.4 in \cite{checos2}, but with some variations along our demonstration. By definition of weak convergence of measures it is sufficient to prove that for a given function $\varphi\in C_0(I^2,\mathbb{R})$, the expression

\begin{equation}\label{weak conv}
\left|\displaystyle\int_{\mathcal{Q}}\varphi d\nu_{n_k}-\displaystyle\int_{\widehat{f}_{n_k}(F^{-1}(\mathcal{Q}))}\varphi d(\widehat{f}_{n_k}|_{F^{-1}(\mathcal{Q})})_{\#}\mu_{n_k}\right|
\end{equation}

tends to $0$ when $k$ goes to $+\infty$. Observe that by the triangle inequality, the expression (\ref{weak conv}) can be bounded from above by

\begin{equation}\label{weak conv 2}
\underbrace{\left|\displaystyle\int_{\mathcal{Q}}\varphi d\nu_{n_k}-\displaystyle\int_{\widehat{f}_{n_k}(F^{-1}(\mathcal{Q}))}\varphi d\nu_{n_k}\right|}_{=T_{1,n_{k}}}+\underbrace{\left|\displaystyle\int_{\widehat{f}_{n_k}(F^{-1}(\mathcal{Q}))}\varphi d\nu_{n_k}-\displaystyle\int_{\widehat{f}_{n_k}(F^{-1}(\mathcal{Q}))}\varphi d(\widehat{f}_{n_k}|_{F^{-1}(\mathcal{Q})})_{\#}\mu_{n_k}\right|}_{=T_{2,n_k}};
\end{equation}

we denote by $T_{1,n_k}$ and $T_{2,n_k}$ the first and the second term in (\ref{weak conv 2}), respectively.\\

Notice that $T_{1,n_k}$ is at most $||\varphi||_{\infty}\nu_{n_k}(\mathcal{Q}\Delta\hat{f}_{n_k}(F^{-1}(\mathcal{Q})))$. To show $T_{1,n_k}\longrightarrow 0$, observe that

\begin{equation*}
\begin{split}
\mathcal{Q}\Delta\widehat{f}_{n_k}(F^{-1}(\mathcal{Q}))&\subset\bigcup_{i=1}^N  F(\mathcal{H}_i)\Delta\widehat{f}_{n_k}(\mathcal{H}_i)\\
       &\subset\bigcup_{i=1}^N\overline{B}(\partial F(\mathcal{H}_i),||F-\widehat{f}_{n_k}||_{\infty}),
\end{split}
\end{equation*}

where in the second step we use Lemma \ref{sym diff} for $U=\mathcal{H}_i$. Thus, by using the weak convergence of $\nu_{n_k}$ to $\lambda$, it follows that

\begin{equation*}
\nu_{n_k}(\mathcal{Q}\Delta\widehat{f}_{n_k}(F^{-1}(\mathcal{Q})))\leq\nu_{n_k}(\overline{B}(\partial\mathcal{Q},||F-\widehat{f}_{n_k}||_{\infty}))\longrightarrow 0. 
\end{equation*}

Therefore, the first term in (\ref{weak conv 2}) converges to 0.

\medskip

To prove that $T_{2,n_k}$ tends to 0 we firstly note that this expression can be bounded from above by

\begin{equation}\label{boundT2}
\dfrac{||\varphi||_{\infty}}{l_{n_k}^2} |A_{n_k}|,
\end{equation}

where $A_n:=\widehat{f}_n(\mathcal{H})\cap\frac{1}{l_n}\mathbb{Z}^2\setminus f_n(\mathcal{H})$. Observe that $\widehat{f}_{n_k}(\mathcal{H})$ is contained in $B(\mathcal{Q},||F-\widehat{f}_{n_k}||_{\infty})$ and thus 

\begin{equation}\label{loss of mass subset 1}
A_{n_k}\subset\left(B(\mathcal{Q},||F-\widehat{f}_{n_k}||_{\infty})\cap\frac{1}{l_{n_k}}\mathbb{Z}^2\right)\setminus f_{n_k}(\mathcal{H}\cap B(||x_{k,j(n_k)}||)),
\end{equation}

where the set $\mathcal{H}\cap\mathcal{R}_{n_k}\cap B(||x_{k,j(n_k)}||)$ is non-empty for a large-enough positive integer $k$, as a consequence of Lemmas \ref{claim4} and \ref{H=Q}. Thus by the injectivity of $f_{n_k}$ and by the part i) of Lemma \ref{conv lemma 2} which says that $\mathcal{H}\cap\mathcal{R}_{n_k}$ is contained in $f_{n_k}^{-1}(\mathcal{Q}_{k,0})\cap I^2$ for a sufficiently large $k$, we get

\begin{equation}\label{loss of mass subset 2}
\begin{split}
f_{n_k}(\mathcal{H}\cap B(||x_{k,j(n_k)}||))=f_{n_k}(f_{n_k}^{-1}(\mathcal{Q}_{k,0})\cap B(||x_{k,j(n_k)}||))\setminus f_{n_k}(f_{n_k}^{-1}(\mathcal{Q}_{k,0})\cap B(||x_{k,j(n_k)}||)\setminus\mathcal{H})
\end{split}
\end{equation}

(recall that by the choice of $k$ in \eqref{k},  the set $f_{n_k}^{-1}(\mathcal{Q}_{k,0})\cap\mathcal{R}_{n_k}$ is contained $F^{-1}(W)$). Thus, from \eqref{loss of mass subset 1}, \eqref{loss of mass subset 2} and Lemma \ref{claim4}, after taking cardinality and dividing by $l_{n_k}^2$ we get:

\begin{equation}\label{weak con 3}
\begin{split}
\frac{|A_{n_k}|}{l_{n_k}^2}\leq &\frac{1}{l_{n_k}^2}|f_{n_k}(f_{n_k}^{-1}(\mathcal{Q}_{k,0})\cap B(||x_{k,j(n_k)}||)\setminus\mathcal{H})|+\nu_{n_k}(B(\mathcal{Q},||F-\widehat{f}_{n_k}||_{\infty})\setminus\mathcal{Q}_{k,0})\\
&+\nu_{n_k}(\mathcal{Q}_{k,0}\setminus\mathcal{Q}_{k,j(n_k)}),
\end{split}
\end{equation}

where $B(\mathcal{Q},||F-\widehat{f}_{n_k}||_{\infty})\setminus\mathcal{Q}_{k,0}$ can be empty or not. By Lemma \ref{conv lemma 2} and from the choice of $k$ in \eqref{k}, the first term in the right-hand side of \eqref{weak con 3} can be estimated as below:

\begin{equation*}
\begin{split}
\frac{1}{l_{n_k}^2}|f_{n_k}(f_{n_k}^{-1}(\mathcal{Q}_{k,0})\cap B(||x_{k,j(n_k)}||)\setminus\mathcal{H})|&=\frac{1}{l_{n_k}^2}|f_{n_k}^{-1}(\mathcal{Q}_{k,0})\cap B(||x_{k,j(n_k)}||)\cap\mathcal{R}_{n_k}\setminus\mathcal{H}|\\
&\leq \mu_{n_k}\left(B\left(\partial\mathcal{H},\frac{1}{k}\right)\right)\\
&\longrightarrow\rho\lambda (\partial\mathcal{H})=0
\end{split}
\end{equation*}

If $B(\mathcal{Q},||F-\widehat{f}_{n_k}||_{\infty})\subset\mathcal{Q}_{k,0}$, then the second term in \eqref{weak con 3} is zero. In the other possible case, i.e if $\mathcal{Q}_{k,0}\subset B(\mathcal{Q},||F-\widehat{f}_{n_k}||_{\infty})$, we get

\[
\nu_{n_k}(B(\mathcal{Q},||F-\widehat{f}_{n_k}||_{\infty})\setminus\mathcal{Q}_{k,0})\leq\nu_{n_k}\left(B\left(\mathcal{Q},||F-\widehat{f}_{n_k}||_{\infty}+\frac{1}{k}\right)\right)\longrightarrow 0
\]

Finally, to estimate the most delicate term in \eqref{weak con 3}, we observe that:

\begin{equation*}
\begin{split}
\nu_{n_k}(\mathcal{Q}_{k,0}\setminus\mathcal{Q}_{k,j(n_k)})&\leq \lambda\left(B\left(\mathcal{Q}_{k,0}\setminus\mathcal{Q}_{k,j(n_k)},\frac{1}{l_{n_k}}\right)\right)\\
&\leq\sum_{j=1}^{j(n_k)}\lambda(\mathcal{Q}_{k,j-1}\setminus\mathcal{Q}_{k,j})+\lambda\left(B\left(\partial\mathcal{Q}_{k,0},\frac{1}{l_{n_k}}\right)\right)+\lambda\left(B\left(\partial\mathcal{Q}_{k,j(n_k)},\frac{1}{l_{n_k}}\right)\right)\\
&\leq \frac{L}{l_{n_k}}\left(2l_{\mathcal{Q}}+\frac{2}{k}-\frac{L}{l_{n_k}}\right)j(n_k)+\lambda\left(B\left(\partial\mathcal{Q}_{k,0},\frac{1}{l_{n_k}}\right)\right)\\
&\ +\lambda\left(B\left(\partial\mathcal{Q}_{k,j(n_k)},\frac{1}{l_{n_k}}\right)\right),
\end{split}
\end{equation*}

where the bound $l_{\mathcal{Q}}$ denotes the side-length of $\mathcal{Q}$. On one hand, it is direct that

\begin{equation*}
\lim_{k\to\infty}\lambda\left(B\left(\partial\mathcal{Q}_{k,0},\frac{1}{l_{n_k}}\right)\right)=\lim_{k\to\infty}\lambda\left(B\left(\partial\mathcal{Q}_{k,j(n_k)},\frac{1}{l_{n_k}}\right)\right)=0.
\end{equation*}

On the other hand, by the construction of $x_{k,j(n_k)}$ and by Lemma \ref{claim3} we have that $j(n_k)\leq ||x_{k,0}||$. Moreover, by Remark \ref{remcentpoint} and since $f$ is $\omega$-co-uniformly continuous, there holds that

\begin{equation*}
||x_{k,0}||\leq\frac{\omega\left(\left(\frac{l_{\mathcal{Q}}}{2}+\frac{1}{k}\right)l_{n_k}\right)}{l_{n_k}}+1.
\end{equation*}

Hence, since the co-uniformity is of order $o(r^2)$, we have that 

\begin{equation*}
\frac{L}{l_{n_k}}\left(2l_{\mathcal{Q}}+\frac{2}{k}-\frac{L}{l_{n_k}}\right)j(n_k)\leq L\left(2l_{\mathcal{Q}}+\frac{2}{k
}-\frac{L}{l_{n_k}}\right)\left(\frac{\omega\left(\left(\frac{l_{\mathcal{Q}}}{2}+\frac{1}{k}\right)l_{n_k}\right)}{l_{n_k}^2}+\frac{1}{l_{n_k}}\right)\longrightarrow 0.
\end{equation*}

Thus \eqref{boundT2} converges to 0 and therefore $T_{2,n_k}\longrightarrow 0$ when $k\to\infty$, as desired.
\end{proof}

\vspace{0.2cm}

\noindent{\bf Acknowledgments}. This work was funded by the ANID/FONDECYT Postdoctoral Grant 3210109. I would like to thank Andr\'es Navas for the enlightening discussions and Mircea Petrache for his comments. I also thank the anonymous referee for carefully reading this paper and helpful comments.

%%%%%%%%%%%%%%%%%%%%%%%%%%%%%%%%%%%%%%%%%%%%%%%%%%%%%%%%%%%%%%%%%%%%%%

\begin{small}

\vspace{0.15cm}

\noindent Rodolfo Viera

\noindent Facultad de Matem\'aticas

\noindent Pontificia Universidad Cat\'olica de Chile

\noindent Av. Vicuña Mackenna 4860, Macul, Santiago, Chile

\noindent Email: rodolfo.viera@mat.uc.cl

\vspace{0.2cm}

\end{small}

\end{document}